\documentclass[11pt]{article}
\usepackage{tikz}
\usetikzlibrary{patterns}
\usetikzlibrary{decorations.pathreplacing}
\usepackage[outline]{contour}
\usepackage{lmodern}
\contourlength{2pt}

\usepackage{amsthm,amsmath}
\usepackage[numbers]{natbib}
\RequirePackage[citecolor=blue,urlcolor=blue]{hyperref}
\usepackage{amssymb}
\usepackage{mathtools}
\usepackage{xfrac,enumerate}
\usepackage{xcolor}

\newcommand{\R}{\mathbb{R}}
\newcommand{\Rd}{\mathbb{R}^{n_1 \times \ldots \times n_d}}
\newcommand{\bigo}{\mathcal{O}}
\newcommand{\N}{\mathcal{N}}
\let\P\undefined
\newcommand{\P}{\mathcal{P}}
\newcommand{\calS}{\mathcal{S}}

\newcommand{\M}{\mathcal{M}}
\newcommand{\hf}{{\hat f}}
\newcommand{\rf}{{\rm f}}
\newcommand{\rmq}{{\rm q}}
\newcommand{\rp}{{\rm p}}
\newcommand{\tf}{{\tilde \phi}}
\newcommand{\tF}{{\tilde \Phi}}

\newcommand{\tS}{{\tilde S}}

\newcommand{\barf}{\bar f}
\newcommand{\bary}{\bar Y}
\newcommand{\bare}{\bar \epsilon}
\newcommand{\rw}{{\rm w}}

\newcommand{\ttimes}[1]{{\times}_{#1}}

\newcommand{\rem}{{\rm rem}}
\newcommand{\TV}{{\rm TV}}
\newcommand{\tv}{{\tilde v}}
\DeclarePairedDelimiterX{\norm}[1]{\lVert}{\rVert}{#1}
\DeclarePairedDelimiterX{\abs}[1]{\lvert}{\rvert}{#1}
\DeclarePairedDelimiterX{\0norm}[1]{\lVert}{\rVert_{0}}{#1}
\DeclarePairedDelimiterX{\1norm}[1]{\lVert}{\rVert_{1}}{#1}
\DeclarePairedDelimiterX{\2norm}[1]{\lVert}{\rVert_{2}}{#1}
\DeclarePairedDelimiterX{\nnorm}[1]{\lVert}{\rVert_{n}}{#1}
\DeclarePairedDelimiterX{\2nnorm}[1]{\lVert}{\rVert_{n}^2}{#1}
\DeclareMathOperator*{\argmin}{arg\,min}

\newtheorem{definition}{Definition}[section]
\newtheorem{lemma}[definition]{Lemma}
\newtheorem{theorem}[definition]{Theorem}

\theoremstyle{definition}

\newtheorem*{remark}{Remark}

\usepackage{Sweave}
\begin{document}
\Sconcordance{concordance:VitaliTV.tex:VitaliTV.Rnw:%
1 90 1 1 0 27 1}
\Sconcordance{concordance:VitaliTV.tex:./VitaliTV-Section1.Rnw:ofs 119:%
1 97 1}
\Sconcordance{concordance:VitaliTV.tex:./VitaliTV-Section2.Rnw:ofs 217:%
1 101 1}
\Sconcordance{concordance:VitaliTV.tex:./VitaliTV-Section3.Rnw:ofs 319:%
1 93 1}
\Sconcordance{concordance:VitaliTV.tex:./VitaliTV-Section4.Rnw:ofs 413:%
1 73 1}
\Sconcordance{concordance:VitaliTV.tex:./VitaliTV-Section5.Rnw:ofs 487:%
1 574 1}
\Sconcordance{concordance:VitaliTV.tex:./VitaliTV-Section6.Rnw:ofs 1062:%
1 194 1}
\Sconcordance{concordance:VitaliTV.tex:./VitaliTV-Section7.Rnw:ofs 1257:%
1 105 1}
\Sconcordance{concordance:VitaliTV.tex:./VitaliTV-Section8.Rnw:ofs 1363:%
1 62 1}
\Sconcordance{concordance:VitaliTV.tex:./VitaliTV-Section9.Rnw:ofs 1426:%
1 5 1}
\Sconcordance{concordance:VitaliTV.tex:VitaliTV.Rnw:ofs 1432:%
128 36 1}
\Sconcordance{concordance:VitaliTV.tex:./VitaliTV-Section10.Rnw:ofs 1469:%
1 329 1}
\Sconcordance{concordance:VitaliTV.tex:VitaliTV.Rnw:ofs 1799:%
166 1 1}

\title{\bf Tensor denoising with trend filtering}

\author{Francesco Ortelli and Sara van de Geer\\
Seminar f\"{u}r Statistik, ETH Z\"{u}rich\\
R\"{a}mistrasse 101, CH-8092 Z\"{u}rich\\
$\{$fortelli,geer$\}$@ethz.ch}
\date{\today}

\maketitle

\begin{abstract}
We extend the notion of trend filtering to tensors by considering the $k^{\rm th}$-order Vitali variation -- a discretized version of the integral of the absolute value of the $k^{\rm th}$-order total derivative. We prove adaptive $\ell^0$-rates and not-so-slow $\ell^1$-rates for tensor denoising with trend filtering.

For $k=\{1,2,3,4\}$ we prove that the $d$-dimensional margin of a $d$-dimensional tensor can be estimated at the $\ell^0$-rate $n^{-1}$, up to logarithmic terms, if the underlying tensor is a product of $(k-1)^{\rm th}$-order polynomials on a constant number of hyperrectangles. For general $k$ we prove the $\ell^1$-rate of estimation $n^{- \frac{H(d)+2k-1}{2H(d)+2k-1}}$, up to logarithmic terms, where $H(d)$ is the $d^{\rm th}$ harmonic number.

Thanks to an ANOVA-type of decomposition we can apply these results to the lower dimensional margins of the tensor to prove bounds for denoising the whole tensor. Our tools  are interpolating tensors to bound the effective sparsity for $\ell^0$-rates, mesh grids for $\ell^1$-rates and, in the background, the projection arguments by \citet{dala17}.

\end{abstract}

\tableofcontents
\section{Introduction}\label{vtv.s.0}

Let $f^0 \in \Rd$ be a $d$-dimensional tensor  with $n=n_1 \cdot \ldots \cdot n_d$ entries. We want to prove error bounds for tensor denoising, which is the task of recovering $f^0$ from its noisy version $Y=f^0+ \epsilon$, where $\epsilon$ has i.i.d. Gaussian entries with mean 0 and variance $\sigma^2$.

We show that we can estimate the underlying tensor $f^0$ in an adaptive manner with a regularized  least-squares signal approximator. As regularizer we propose the Vitali variation of the $(k-1)^{\rm th}$-order total differences of the candidate estimator for $k \ge 1$. We call this regularizer the ``$k^{\rm th}$-order Vitali total variation''. We use the abbreviation TV for ``total variation''. This approach extends the idea of ``trend filtering'' \citep{kim09-1,tibs14} to tensors. 

We expose the notion of TV regularization, review the literature on adaptive results for TV regularization, explain the concept of adaptation for structured problems, introduce an ANOVA-type of decomposition of a tensor, outline our contributions and finally present the organization of the paper.

\subsection{TV regularization}

A regularized (least-squares) signal approximator is an estimator ${\hat f}$ defined as
$$ \hf := \argmin_{\rf \in \Rd} \left\{ \norm{Y-\rf}^2_2/n + 2 \lambda\ {\rm pen}(\rf)  \right\},$$
where $\norm{\cdot}^2_2$ denotes the sum of the squared entries of its argument, $\lambda>0$ is a tuning parameter and ${\rm pen}(\rf)$ is a regularization penalty.

When ${\rm pen}(\rf)= \norm{D \rf}_1$ for a linear operator $D$ and for $\norm{\cdot}_1$ denoting the sum of the absolute values of the entries of its argument, the regularized signal approximator is called ``$\ell^1$-analysis estimator'' or simply ``analysis estimator'' \citep{elad07}. If the linear operator $D$ is a difference operator, then ${\rm pen}(\rf)= \1norm{D \rf}$ is usually called TV of $\rf$ and the estimator $\hf$ is called TV regularized estimator. Different choices of the difference operator $D$ are possible, resulting in different notions of TV.

For a continuous image defined on $(x_1, \ldots, x_d) \in [0,1]^d$, one can choose $D$ as a discretized version of either the total $k^{\rm th}$-order derivative operator $\prod_{i=1}^d \partial^k/(\partial x_i)^k$ or of the sum of $k^{\rm th}$-order partial derivative operators $\sum_{i=1}^d \partial^k/(\partial x_i)^k$.

\subsection{Literature review: adaptive results for TV regularization}

For $d=1$ partial and total derivatives coincide. With $D$ being the first order difference matrix, the TV regularized estimator is also known under the name ``fused Lasso'' \citep{tibs05,frie07}. Adaptivity of the fused Lasso has been proved by \citet{dala17, lin17b,gunt20}.

The ``edge Lasso'' extends the fused lasso to graphs and is studied by \citet{shar12,hutt16}. \citet{orte18,orte19-2} prove adaptivity of the edge Lasso on tree graphs and cycle graphs, respectively. 

The idea of the fused Lasso can also be extended to the penalization of higher-order differences. This extension is called ``trend filtering'' \citep{kim09-1, tibs14, tibs20}. Adaptivity of trend filtering is established in \citet{gunt20,vand19}. \citet{wang16} consider trend filtering on graphs,  \citet{sadh17b} in higher-dimensional situations and \citet{sadh19} for additive models.

Here, we consider the case of $D$ being a discretization of $\prod_{i=1}^d \partial^k/(\partial x_i)^k$. We call the corresponding notion of TV ``$k^{\rm th}$-order Vitali TV''. In the literature, signal approximators regularized with the Vitali TV are studied by \citet{mamm97-2,orte19-4,fang19}.
\citet{orte19-4} prove adaptivity for $d=2$ and $k=1$.
\citet{fang19} show adaptivity for $d=2$ and $k=1$ using as regularizer the Hardy-Krause variation, which is the sum of the Vitali TV of a matrix and of its margins. In this paper we will prove adaptivity of tensor denoising with $k^{\rm th}$-order Vitali TV regularization for $k=\{1,2,3,4\}$ and general dimension $d\ge 1$.  The results obtained for $k=\{1,2,3,4\}$ and $d=1$ in \citet{vand19} and for $k=1$ and $d=2$ in \citet{orte19-4} will then be retrieved as special cases.

Signal approximators regularized with $D$ being a discretization of the partial derivatives $\sum_{i=1}^d \partial^k/(\partial x_i)^k$ are studied by \citet{hutt16,sadh17b} for general $d$. For $d=2$, \citet{chat19} show the fast rate $n^{-3/4}$ for estimating axis-aligned rectangles.

\subsection{Adaptation for structured problems}

The analysis estimator $ \hf$ can be recast in a constructive formulation as ``synthesis  estimator''. One can find  dictionary tensors $\{\phi_j\in \Rd\}_{j \in [p]}$, such that
$$ \hf= \sum_{j=1}^p {\hat \beta}_j \phi_j , \text{ where } {\hat \beta}:= \argmin_{b \in \R^n} \left\{ \norm{Y-\sum_{i=1}^p b_j \phi_j}^2_2/n + 2 \lambda \sum_{j \not\in U} \abs{b_j}  \right\},$$
and $U \subseteq \{1, \ldots, p\}$ is a set of indices, cf. \citet{elad07}. The Lasso estimator \citep{tibs96,buhl11,vand16} is an instance of synthesis estimator. The dictionary $\{\phi_j\}_{j \in [p]}$  and  the set of unpenalized coefficients $U \subseteq [p]$ depend on $D$. We can see that $D$ imposes structure on the estimator: it determines the dictionary with which the estimator is constructed. For instance, in the case of the $1^{\rm st}$-order Vitali TV, the dictionary $\{\phi_j\}_{j \in [p]}$  consists of tensors being constant on hyperrectangles. Therefore, the estimator $\hf$ is constant on few hyperrectangular pieces.

Our goal is to prove adaptation of the estimator $\hf$ to the underlying signal $f^0$, when $\1norm{D\rf}$ is the $k^{\rm th}$-order Vitali TV.

Adaptation is a consequence of a high-probability upper bound on the mean squared error (MSE) in the form of the oracle inequality
\begin{equation}\label{tensor.adaptivity}
\norm{\hf-f^0}^2_2/n \le \norm{g-f^0}^2_2/n + \rem(D,g,S),
\end{equation}
where $g \in \Rd$ is an  arbitrary tensor, $S$ is an  arbitrary set of indices of $D g$ and $\rem(D,g,S)$ is a remainder term.
A result of the form of \eqref{tensor.adaptivity} establishes the adaptation of the estimator $\hf$, provided that the remainder term $\rem(D, g=f^0,S=S^0)$ converges to zero, where $S^0$ is the set of the indices of the  nonzero coefficients of $D f^0$. The cardinality $s^0:= \abs{S^0}$ of $S^0$ is called the ``sparsity'' of $f^0$ with respect to $D$.

We can optimize the upper bound in \eqref{tensor.adaptivity} over $g$ and $S$. However, the optimizers $g^*$ and $S^*$ will depend on $f^0$ -- which is unobserved. Hence the name ``oracle'' for the pair $(g^*(f^0), S^*(f^0))$ and the name ``oracle inequality'' for results as \eqref{tensor.adaptivity}.

Such a result is considered to be adaptive, since different underlying true tensors $f^0$ will possibly give place to different oracles $(g^*(f^0), S^*(f^0))$ and to different values for the upper bound.

Results as \eqref{tensor.adaptivity} are only useful if it can be proved that $\rem(D,f^0,S^0)$ converges to zero. Typically 
$$ \rem(D,f^0,S^0)=\bigo \left( \lambda^2\Gamma^2_D(S^0) \right),$$
where $\Gamma^2_D(S^0)$ is called ``effective sparsity'' and depends both on $D$ and $S^0$.
Proving adaptivity therefore translates into proving a bound for the effective sparsity: a task which depends on the structure imposed by $D$. To bound the effective sparsity for tensor denoising with trend filtering we use an interpolating tensor, in analogy to the interpolating vector and the interpolating matrix by \citet{vand19,orte19-4}.

Adaptive results as \eqref{tensor.adaptivity} are a consequence of a careful choice of $\lambda$.
The general theory for the Lasso \citep{buhl11,vand16} suggests the choice $\lambda \asymp \lambda_0 \asymp \sqrt{\log(n)/n}$, where $\lambda_0$ is called the ``universal choice''. The universal choice ensures that all the noise is overruled. However, \citet{dala17} show that also the smaller choice $\lambda \asymp {\tilde \gamma} \lambda_0$ is possible, where ${\tilde \gamma} >0$ is a scaling factor  which accounts for the correlation in the dictionary $\{\phi_j\}_{j \in [p]}$ induced by $D$ and $S^0$. The projection arguments by \citet{dala17} in the background of our results allow us to choose the tuning parameter of smaller order than the universal choice $\lambda_0$.

Projection arguments have been discussed in the literature. We do not report them here but refer instead to Theorem 3 in \citet{dala17}, Lemma B2 and Lemma C2 in \citet{orte19-2}, Lemma 13 in \citet{orte19-4} and to \citet{vand21}.

\subsection{ANOVA decomposition}

In the continuous case, the ``nullspace'' of the ${k^{\rm th}}$-order derivative operator along one coordinate is made of constant, linear, ..., $(k-1)^{\rm th}$-order monomial functions. The nullspace of the total derivative operator in $d$-dimensions is made of $d$-dimensional functions which are  linear, ..., $(k-1)^{\rm th}$-order monomial along at least one coordinate. In the discrete case when $n_1 \asymp \ldots \asymp n_d$ the linear space spanned by such tensors is  $n^{1-1/d}$-dimensional.

We will decompose a tensor $f\in \Rd$ into a sum of mutually orthogonal tensors. Each of these mutually orthogonal tensors will be  constant or linear or ... or $(k-1)^{\rm th}$-order monomial along a set of $l$ coordinates, for $l \in [0:d]$. This construction will be carried out for all possible sets of coordinates in $[d]$. Tensors being constant or linear or ... or $(k-1)^{\rm th}$-order monomial along $d-l$ coordinates will be called $l$-dimensional margins.

We will adaptively estimate $l$-dimensional margins with $l$-dimensional Vitali TV regularized estimators, for $l \in [d]$. The $0$-dimensional margins will be estimated by ordinary least squares at a rate $n^{-1}$. By estimating all the margins adaptively we will be able to prove adaptivity of the denoising of the whole tensor via Vitali TV regularization.

\subsection{Contributions}

Previously, we have derived tools like interpolating vectors and matching derivatives to prove adaptivity for trend filtering ($d=1$ and $k=\{1,2,3,4\}$, see \citet{vand19}). In \citet{orte19-4} we have come up with tools to extend our results for adaptation of the fused Lasso ($d=1$ and $k=1$) to the two-dimensional case of image denoising ($d=2$ and $k=1$). Here, we show in the first place how to combine and extend the tools from image denoising and one-dimensional trend filtering to handle trend filtering for $k=\{1,2,3,4\}$ and for general dimension $d$. Establishing adaptivity requires a so-called ``bound on the antiprojections''. We prove a formula giving the bounds on the antiprojections for general $k$ and $d$. We then propose an ANOVA decomposition to ensure that all the margins of a $d$-dimensional tensor can be estimated adaptively.

Lastly, we prove slow rates for tensor denoising with trend filtering. We extend the idea of mesh grid by \citet{orte19-4} to general $d$ and general $k$. We then prove a bound on the antiprojections with the help of the mesh grid holding for all $d$ and all $k$.

The integration of the arguments by \citet{vand19} with the ones by \citet{orte19-4}, the general bounds on the antiprojections and the ANOVA decomposition allow us to present general risk bounds for tensor denoising with trend filtering.

\subsection{Organization of the paper}

In Section \ref{vtv.s.1} we expose the required notation, the model and define  the trend filtering estimator for the $d$-dimensional margin.

In Section \ref{vtv.s.2} we list our contributions and give a  preview of the results: adaptive $\ell^0$-rates and not-so-slow $\ell^1$-rates.

In Section \ref{vtv.ss.2.3} we derive the synthesis form of the trend filtering estimator for the $d$-dimensional margin.

Proving the main result on adaptivity for tensor denoising with  trend filtering is the topic of Section \ref{vtv.s.3}.

In Section \ref{vtv.s.4} we apply a general result on not-so-slow $\ell^1$-rates for analysis estimators to tensor denoising with trend filtering.

In Section \ref{vtv.s.5} we show the ANOVA decomposition of a tensor and define the estimators for lower-dimensional margins.

In Section \ref{vtv.s.6} we apply the results on adaptivity and on not-so-slow $\ell^1$-rates to the estimators for the lower-dimensional margins  defined in Section \ref{vtv.s.5}. This will establish adaptivity and not-so-slow rates for the estimation of the whole tensor.

Section \ref{TV1S8} concludes the paper.

\section{Model, notation and estimator}\label{vtv.s.1}

We consider the model
$$ Y=f^0+ \epsilon,$$
where $Y, f^0, \epsilon\in \Rd$ are $d$-dimensional tensors and $\epsilon$ has i.i.d. $\N(0, \sigma^2)$ entries with known variance $\sigma \in (0, \infty)$.  For the case of unknown variance we refer to \citet{orte19-2}, who show how to estimate $f^0$ and $\sigma$ at the same time. 

The goal is to estimate $f^0$ given its noisy observations $Y$. We consider a signal approximator regularized with the Vitali TV. 

\subsection{Signals supported on $d$-dimensional tensors}
For two integers $i\le j$ we define $[i:j]:= \{i, \ldots, j \}$. Moreover, if $i=1$ we write $[j]:=[1:j]$.

Let $f\in \mathbb{R}^{n_1\times \ldots \times n_d}$ be a $d$-dimensional tensor with $n:= n_1 \ldots n_d$ entries. For indices $(j_1, \ldots, j_d)\in [n_1]\times \ldots \times [n_d]$ we refer to the corresponding entry of $f$ by $f_{j_1, \ldots, j_d}$ using indices or by $f(j_1, \ldots, j_d)$ using arguments. 


For $(j_1',\ldots, j_d'), (j_1'',\ldots, j_d'')\in [n_1]\times \ldots \times [n_d] $ we use the notation
$$ \sum_{j_1',\ldots, j_d'}^{j_1'',\ldots, j_d''} f_{j_1, \ldots, j_d}:= \sum_{j_d=j_d'}^{j_d''} \cdots \sum_{j_1=j_1'}^{j_1''} f_{j_1, \ldots, j_d}.$$
Similarly we write 
$$\{f_{j_1, \ldots, j_d}\}_{j_1', \ldots, j_d'}^{j_1'', \ldots, j_d''}:= \{f_{j_1, \ldots, j_d}\}_{(j_1, \ldots, j_d)=(j_1',\ldots,  j_d')}^{(j_1'', \ldots, j_d'')}.$$

By $\| f \|_2 := ( \sum_{1, \ldots, 1}^{n_1, \ldots, n_d}  f^2_{j_1, \ldots, j_d}  )^{1/2}$ we denote the Frobenius norm of $f$.
Moreover we define $ \norm{f}_1:=  \sum_{1, \ldots, 1}^{n_1, \ldots, n_d} \abs{f_{j_1, \ldots, j_d}}$ as  the sum of the absolute values of the entries of $f$. 

\subsubsection{Tensors with product structure}
We now let $f \in \Rd$ be a $d$-dimensional tensor with $n:= n_1 \cdot \ldots \cdot n_d$ entries. Define the set of indices $I$ of the entries of $f$ as $ I:= [n_1] \times \ldots \times [n_d]$.

We say that $f$ has product structure if there are vectors $\{f_j\}_{j \in [d]}$ such that
$$ f(j_1, \ldots, j_d)= f_1(j_1) \cdot \ldots \cdot f_d(j_d), \forall (j_1, \ldots, j_d) \in I.$$
We then write $ f= f_1 \times \ldots \times f_d$.

Let $f$ and $g$ be tensors with product structure. We consider the entry-wise multiplication $(f \odot g)_{j_1, \ldots, j_d}= f_{j_1, \ldots, j_d} g_{j_1, \ldots, j_d}, (j_1, \ldots, j_d ) \in I$.

It holds that $ (f \odot g)_{j_1, \ldots, j_d}= \prod_{l=1}^d f_l(j_l)g_l(j_l), \forall (j_1, \ldots, j_d ) \in I$.

\subsubsection{Orthogonality between tensors}
The operation $ \sum_{1, \ldots, 1}^{n_1, \ldots, n_d} (f \odot g)_{j_1, \ldots, j_d}$
is the equivalent of the scalar product for tensors.

We say that the tensors $f$ and $g$ are orthogonal if $ \sum_{1, \ldots, 1}^{n_1, \ldots, n_d} (f \odot g)_{j_1, \ldots, j_d}=0$.
If $f$ and $g$ have product structure and $f_l$ and $g_l$ are orthogonal to each other for at least one coordinate $l\in [d]$, then $f$ and $g$ are orthogonal too.

\subsubsection{Linear subspaces and orthogonal projections}

Let $\mathcal{W}$ be a linear subspace of $\R^{n_1\times \ldots \times n_d}$ and let ${\mathcal W}^{\perp}$ be its orthogonal complement. By ${\rm I}: \R^{n_1\times \ldots \times n_d} \mapsto \R^{n_1\times \ldots \times n_d}$ we denote the identity operator, i.e., ${\rm I}f=f$.  By ${\rm P}_{\mathcal W}$ we denote the orthogonal projection operator onto ${\mathcal W}$ and by ${\rm A}_{\mathcal W}:={\rm I}- {\rm P}_{\mathcal W}= {\rm P}_{{\mathcal W}^{\perp}}$ the corresponding orthogonal antiprojection operator. For a tensor $f\in \R^{n_1\times \ldots \times n_d}$ we write $f_{\mathcal W}:= {\rm P}_{\mathcal W}$ and $f_{{\mathcal W}^{\perp}}:= f-f_{\mathcal W}$.

For a linear operator $\Delta$, let $\N(\Delta)$ denote its nullspace.

\subsection{Estimator}

Let $k$ be an integer in $\{1, \ldots, \min_{i \in [d]} n_i-1\}$.

Let $D^k_i$ be the  $k^{\rm th}$-order  difference operator along the $i^{\rm th}$ coordinate, defined as
$$ (D_i^k f)(j_1,\ldots,  j_i, \ldots, j_{d}):= n_i^{k-1}\sum_{l=0}^{k} (-1)^l \binom{k}{l} f(j_1,\ldots, j_i-l,  \ldots, j_{d}),$$
for $ (j_1,\ldots, j_{i-1}, j_i, j_{i+1}, \ldots, j_{d}) \in [n_1] \times \ldots \times [n_{i-1}] \times [k+1: n_i] \times [n_{i+1}] \times \ldots \times [n_d]$.

\begin{definition}[Total $k^{\rm th}$-order difference operator]\label{vtv.d.2.1}
The total $k^{\rm th}$-order difference operator $D^k$ is defined as
$$D^k:= \prod_{i =1}^d D^k_i.$$
\end{definition}

The total $k^{\rm th}$-order difference operator $D^k$ can be seen as a discretized version of $\prod_{i=1}^d \partial^k/(\partial x_i)^k $. It is important to note that the definition of $D^k$ implicitly includes a factor $n^{k-1}$ that stems from the discretization.

The Vitali TV of a tensor $ f \in \Rd$ is defined as the sum of the absolute values of its total $k^{\rm th}$-order differences.

\begin{definition}[$k^{th}$-order Vitali TV]\label{vtv.d.2.2}
The $k^{th}$-order Vitali TV $\TV_k(f)$ of a $d$-dimensional tensor $f\in \Rd$ is defined as
$$ \TV_k(f):= \1norm{D^kf}.$$
\end{definition}

The $k^{th}$-order Vitali TV has the canonical scaling  $\TV_k(f)= \bigo(1)$ due to the normalization by the factor $n^{k-1}$ in the definition of $D^k$. We refer to \citet{sadh16} for more about canonical scalings.

We define the nullspace $\N_k$ of $D^k$ as $ \N_k:= \{f \in \Rd: D^kf=0\}$
and its orthogonal complement as $\N^{\perp}_k$.
We call $f_{\N^{\perp}_k}$ the $d$-dimensional margin of a tensor $f\in \Rd$.

\begin{definition}[$k^{\rm th}$-order trend filtering estimator]\label{vtv.d.2.3}
 The $k^{\rm th}$-order Vitali trend filtering estimator $\hf_{\N^{\perp}_k}$ for the $d$-dimensional margin $f^0_{\N^{\perp}_k}$ is defined as
 $$ \hf_{\N^{\perp}_k}:= \argmin_{\rf\in \Rd} \left\{ \norm{(Y-\rf)_{\N^{\perp}_k}}^2_2/n+ 2 \lambda \TV_k(\rf) \right\},$$
 where $\lambda>0$ is a tuning parameter.
\end{definition}

\subsection{Active sets}

Let $S\subseteq [3:n_1-1] \times \ldots \times [3:n_d-1]$ be a subset of the indices of $D^k f$ for some tensor $f \in \R^{n_1 \times \ldots \times  n_d}$. We write $s:= \abs{S}$ and $ S=\{t_1, \ldots, t_s\}$, where $t_m=(t_{1,m}, \ldots, t_{d,m})$. We call $\{t_m\}_{m=1}^s$ the jump locations.


Moreover we define $ a_S :=  \{ a_{j_1, \ldots, j_d}, \ {(j_1, \ldots, j_d)\in S} \}$ and $\ a_{-S}:= \{ a_{j_1, \ldots, j_d}, \ (j_1, \ldots, j_d) \notin S \}$. 
We will use the same notation $a_S $ for the tensor
which shares its entries with $a$ for $(j_1, \ldots, j_d) \in S$ and has all its other
entries equal to zero. 
Similarly, we will also denote by $a_{-S}$ a tensor that shares its entries with $a$ for $(j_1, \ldots, j_d) \not \in S$ and has its other entries equal to zero.

\section{Contributions}\label{vtv.s.2}

We make the following contributions:
\begin{itemize}
\item We extend the idea of trend filtering to $d$-dimensional settings via the Vitali variation and total discrete derivatives.
\item We prove adaptive $\ell^0$-rates for tensor denoising with trend filtering for $k=\{1,2,3,4\}$, see Theorem \ref{vtv.t.2.1}, a simplified version of Theorem \ref{vtv.t.3.1}. The rates for $d=1$ and $k=\{1,2,3,4 \}$ and for $d=2$ and $k=1$ are known. Rates for the other cases are new contributions.
We also expose some sufficient conditions to find adaptive bounds for general $k$. For each given  $k$ one can check by computer whether the conditions hold but the problem of showing that they hold for general $k$ remains open.
\item We prove not-so-slow $\ell^1$-rates for tensor denoising with trend filtering, see Theorem \ref{vtv.t.2.2}. Here too, the rates for $d =2$ and $k \ge 2$ and for $d \ge 3$ are new contributions. It is still an open problem whether these rates correspond for $d\ge 2$ to minimax rates (modulo log terms).
\item We extend the idea of ANOVA decomposition from $1^{\rm st}$-order differences to $k^{\rm th}$-order differences in $d$ dimensions. By means of this ANOVA decomposition we can apply the results for the $d$-dimensional margin to lower dimensional margins. We obtain $\ell^0$- and $\ell^1$-rates for the estimation of the whole tensor by trend filtering.
\item Our results allow to recover previous results for trend filtering and image denoising \citep{vand19,orte19-4} as special cases.
\end{itemize}

\subsection{Preview of the results}

We consider tensors in $\Rd$ such that $n_1= \ldots= n_d$.

Let
$$ \lambda_0 (t) :=\sigma \sqrt { 2 \log (2n) +2 t \over n }, t >0 . $$
We call $\lambda_0(t)$ the ``universal choice'' of the tuning parameter.
The universal choice $ \lambda= \lambda_0(t)$ guarantees that all the noise is overruled. However, our results also allow for a smaller choice than the universal choice, due to the projection arguments by \citet{dala17} in the background.

\begin{theorem}[Adaptivity of Vitali trend filtering, simplified]\label{vtv.t.2.1}
 Fix $k\in\{1,2,3,4\} $. Let $g \in \R^{n^{1/d}\times \ldots \times n^{1/d}}$ be arbitrary. Let $S\subseteq \ttimes{i \in [d]}[k+2:n^{1/d}-1]$ be an arbitrary set of size $s:=\abs{S}$ defining a regular grid of cardinality $s^{1/d}\times\ldots \times s^{1/d}$ parallel to the coordinate axes.
For a large enough constant $C>0$ only dependent on $k$, choose
$$\lambda \ge C \ d^{3/2}\  \frac{\lambda_0(\log(2n))}{ s^{\frac{2k-1}{2d}}} .$$
Then, with probability at least $1-1/n$, it holds that
\begin{eqnarray*}
\norm{(\hf-f^0)_{\N^{\perp}_k}}^2_2/n &\le& \norm{g-f^0_{\N^{\perp}_k}}^2_2/n + 4 \lambda \norm{(D^k g)_{-S}}_1\\
&&+ \bigo\left(\lambda^2 \ \frac{s^{2k}\log(n/s)}{n}   \right).
\end{eqnarray*}
\end{theorem}

\begin{proof}
See Subsection \ref{vtv.ss.3.5} for the proof of the more general Theorem \ref{vtv.t.3.1}.
\end{proof}

Some examples of the exponent of $s$ in the rate of Theorem \ref{vtv.t.2.1} for $d=\{1,2,3\}$ and $k=\{1,2,3,4\}$  are exposed in Table \ref{vtv.tab.2.1}.
\begin{table}[h]
\center
\begin{tabular}{|l|l|l|l|l|l|}
\hline
 & $k=1$ & $k=2$ & $k=3$ & $k=4$  \\ \hline
 $d=1$ & 1 & 1 & 1 & 1  \\ \hline
 $d=2$ & 3/2 & 5/2 & 7/2 & 9/2  \\ \hline
 $d=3$ & 5/3 & 3 & 13/3  & 17/3  \\ \hline
  $d$ general &  $2-1/d$ & $4-3/d$ & $6-5/d$  & $8-7/d$  \\ \hline
\end{tabular}
\caption{Some examples of the exponent of $s$ in the rate of Theorem \ref{vtv.t.2.1} for the choice $\lambda \asymp s^{-\frac{2k-1}{2d}} \lambda_0(\log(2n))$.}\label{vtv.tab.2.1}
\end{table}

If in Theorem \ref{vtv.t.2.1} we set $g=f^0_{\N^{\perp}_k}$ and choose the tuning parameter $\lambda \asymp s^{-\frac{2k-1}{2d}} \lambda_0(\log(2n))$ depending on the (typically unknown) true active set $S_0$, 
we obtain the rate
$$ \bigo\left(\frac{s_0^{2k-\frac{2k-1}{d}}\log n \log(n/s)}{n}   \right).$$
If in Theorem \ref{vtv.t.2.1} we set $g=f^0_{\N^{\perp}_k}$ and we choose the tuning parameter $\lambda \asymp \lambda_0(\log(2n))$ in a completely data-driven way not depending on the (typically unknown) true active set $S_0$,  we obtain the rate
$$ \bigo\left(\frac{s_0^{2k}\log n \log(n/s)}{n}   \right).$$

We now fix $k\in [1:\min_{i \in [d]}n_i-1]$. For $d \in \mathbb{N}$ define the  $d^{\rm th}$ harmonic number $H(d)$ as $ H(d):= \sum_{i=1}^d {1}/{i}$.

\begin{theorem}[Not-so-slow $\ell^1$-rates for Vitali trend filtering]\label{vtv.t.2.2}
Let $g \in \R^{n^{1/d}\times \ldots \times n^{1/d}}$ be any tensor such that $\TV_k(g)= \bigo(1)$.
Choose
$$\lambda \asymp n^{- \frac{H(d)+2k-1}{2H(d)+2k-1}} \log^{\frac{H(d)}{2H(d)+2k-1}}(n).$$
Then, with probability at least $1-\Theta(1/n)$, it holds that
\begin{equation*}
\norm{(\hf-f^0)_{\N^{\perp}_k}}^2_2/n \le \norm{g-f^0_{\N^{\perp}_k}}^2_2/n+ \bigo\left( n^{- \frac{H(d)+2k-1}{2H(d)+2k-1}} \log^{\frac{H(d)}{2H(d)+2k-1}}(n) \right).
\end{equation*}
\end{theorem}

\begin{proof}
See Subsection \ref{vtv.pt.2.2}.
\end{proof}

Some examples of the exponent of $n$ in the rate of Theorem \ref{vtv.t.2.2} for $d=\{1,2,3\}$ and $k=\{1,2,3\}$  are exposed in Table \ref{vtv.tab.2.2}.

\begin{table}[h]
\center

\begin{tabular}{|l|l|l|l|l|}
\hline
 & $k=1$ & $k=2$ & $k=3$  & $k$ general \\ \hline
 $d=1$ & $-2/3$ & $-4/5$ & $-6/7$ & $-2k/(2k+1)$ \\ \hline
 $d=2$ & $-5/8$ & $-3/4$ & $-13/16$  & $-(4k+1)/(4k+4)$ \\ \hline
 $d=3$ & $-17/28$ & $-29/40$ & $-41/52$  &  $- (12k+5)/(12k+16)$ \\ \hline
\end{tabular}
\caption{Some examples of the exponent of $n$ in the rate of Theorem \ref{vtv.t.2.2}.}\label{vtv.tab.2.2}
\end{table}

\section{Synthesis form}\label{vtv.ss.2.3}

According to Definition \ref{vtv.d.2.3}, the  trend filtering estimator is an analysis estimator. In this section we want to rewrite it in a constructive form, that is, in synthesis form. We show that the  trend filtering estimator can be constructed as a linear combination of tensors with product structure, where the factors are truncated monomials of order $k-1$. We call the collection of such tensors the ``dictionary''.

We first define the dictionary and then show that it is the right dictionary to construct the trend filtering estimator.

We start with the one-dimensional case. We then obtain the $d$-dimensional dictionary from the one-dimensional dictionary by constructing tensors with product structure.

\subsection{Dictionary for $d=1$}

Let $ \phi^1_j:=\{1_{\{j'\ge j \}}\}_{j'\in [n]}, j \in [n]$.
The vectors $\{\phi^1_j\}_{j \in [n]}$ are linearly independent and piecewise  constant.

For $2 \le k \le n-1$ define recursively
$$ \phi^k_j:= \begin{cases} \phi^j_j, & j \in [k-1],\\
\sum_{l \ge j} \phi^{k-1}_l/n, & j \in [k:n]. \end{cases}$$
We call the collection $\Phi^k=\{\phi^k_j\}_{j \in [n]}$ the ``original'' dictionary.

The dictionary $\Phi^k$ is a collection of $n$ linearly independent discrete (truncated) monomials: the first $k$ are monomials of order $0, 1, \ldots, k-1$, while the last $n-k$ are truncated monomials of order $k-1$.

We now define a partially orthonormalized version of the dictionary $\Phi^k$, $k \in [n-1]$.

\begin{definition}[Partially orthonormalized dictionary in one dimension]\label{vtv.d.2.5}
The (partially orthonormalized) dictionary $\tF^k=\{\tf^k_j\}_{j \in [n]}$ is defined as
\begin{equation*} {\tf}^k_j:=\begin{cases} \sqrt{n} {\rm A}_{\{\phi^l_l,l \in [j-1]\}} \phi^j_j / \norm{{\rm A}_{\{\phi^l_l,l \in [j-1]\}} \phi^j_j}_2  , & j \in [k],\\
{\rm A}_{\{\phi^l_l, l \in [k]\}} \phi^k_j, & j \in [k+1:n].
\end{cases}
\end{equation*}
\end{definition}

For $k \in [n-1]$, ${\tF}^k=\{\tf^k_j\}_{j \in [n]}$ is again a collection of $n$ linearly independent vectors, where $ \tf^k_1, \ldots, \tf^k_k, \{\tf^k_j \}_{j \in[k+1:n]}$ are mutually orthogonal. Moreover $\norm{\tf^k_j}^2_2= n, j \in [k]$.

\begin{lemma}[Relation between dictionary and difference operator]\label{vtv.l.2.1}
Fix $k \in [n-1]$. It holds that
$$ D^k \phi^k_j= D^k \tf^k_j= \begin{cases}  0, & j \in [k],\\
1_{\{j\}}, & j \in [k+1:n]. \end{cases}$$
\end{lemma}

\begin{proof}
See Appendix \ref{vtv.pl.2.1}.
\end{proof}

As a consequence of Lemma \ref{vtv.l.2.1}, $\{\tf^j_j \}_{j \in [k]}$ span $ \N_k$ and $\{\tf^j_j\}_{j \in [k]}$ is an orthogonal basis for $\N_k$. Moreover $\{ \tf^k_j\}_{j \in [k+1:n]}$ span $ \N^{\perp}_k$.

By Lemma \ref{vtv.l.2.1} combined with Lemma 2.2 in \citet{orte19-1} about the Moore-Penrose pseudoinverse we obtain for the pseudoinverse $(D^k)^+$ that $ (D^k)^+= \{\tf^k_j\}_{j \in [k+1:n]}$.

With the dictionary $\tF^k$  and some coefficients $\{\beta_j\}_{j=k+1}^n$ we can write a vector $f_{\N^{\perp}_k}\in \N^{\perp}_k$ as $ f_{\N^{\perp}_k}=(D^k)^+ \beta $.
Then $ \beta= D^k f_{\N^{\perp}_k}$.

For $d=1$ we therefore obtain the following synthesis form of the estimator $\hf_{\N^{\perp}_k}$:
$$  \hf_{\N^{\perp}_k}= \sum_{j=k+1}^n \tf^k_j \hat{\beta}_j,$$
where
$$ \hat{\beta} = \argmin_{b \in \R^{n-k}} \left\{\norm{Y_{\N^{\perp}_k}- \sum_{j=k+1}^n b_j \tf^k_j }^2_2/n + 2 \lambda \1norm{b}  \right\}.$$

\subsection{Dictionary for general $d$}
Hereafter we fix $k \in [1: \min_{l \in [d]} n_l-1]$.

\begin{definition}[Partially orthonormalized dictonary in $d$-dimensions]\label{vtv.d.2.6}
The dictionary $\{\tf^k_{j_1, \ldots, j_d} \in \Rd\}_{1, \ldots, 1}^{n_1, \ldots, n_d}$ is defined as 
$$ \tf^k_{j_1, \ldots, j_d}= \tf^k_{j_1} \times \ldots \times \tf^k_{j_d},\ (j_1, \ldots, j_d) \in \ttimes{i \in [d]}[n_i].$$
\end{definition}

The dictionary $\{\tf^k_{j_1, \ldots, j_d}\}_{1, \ldots, 1}^{n_1, \ldots, n_d}$ is a collection of $d$-dimensional tensors with product structure.
By Lemma \ref{vtv.l.2.1} and the product structure, $\N_k^{\perp}= {\rm span}(\{\tf^k_{j_1, \ldots, j_d}\}_{k+1, \ldots, k+1}^{n_1, \ldots, n_d} )$.

For a tensor of coefficients $\{\beta_{j_1, \ldots, j_d}\}_{k+1, \ldots, k+1}^{n_1, \ldots, n_d}$, write
$$ f_{\N^{\perp}_k}= \sum_{k+1, \ldots, k+1}^{n_1, \ldots, n_d} \beta_{j_1, \ldots, j_d} \tf^k_{j_1, \ldots, j_d}.$$
Because of the product structure of $\tf^k_{j_1, \ldots, j_d}$ it holds that
$$ D^k f_{\N^{\perp}_k}= \sum_{k+1, \ldots, k+1}^{n_1, \ldots, n_d} \beta_{j_1, \ldots, j_d} (1_{\{j_1\}}\times \ldots \times 1_{\{j_d\}})= \beta.$$
From the fact that any candidate estimator has to belong to the space spanned by $Y_{\N^{\perp}_k}$, it follows that
$$ \hf_{\N^{\perp}_k}= \sum_{k+1, \ldots, k+1}^{n_1, \ldots, n_d} {\hat \beta}_{j_1, \ldots, j_d} \tf^k_{j_1, \ldots, j_d},$$
where
$$ \hat{\beta}=\argmin_{b \in \R^{(n_1-k)\times \ldots \times (n_d-k)}} \left\{ \norm{ Y_{\N^{\perp}_k}- \sum_{k+1, \ldots, k+1}^{n_1, \ldots, n_d} b_{j_1, \ldots, j_d} \tf^k_{j_1, \ldots, j_d}}_2^2/n+ 2 \lambda \1norm{b} \right\}.$$
The synthesis form of the estimator $\hf_{\N^{\perp}_k}$ is useful in two ways. Firstly, to determine the structure of the estimator by specifying the dictionary used to construct it. In our case, $\hf_{\N^{\perp}_k}$ is a linear combination of $d$-dimensional products of $(k-1)^{\rm th}$-order polynomials. Secondly, the dictionary  facilitates the approximation of some orthogonal projections in the proof of adaptive $\ell^0$-rates and not-so-slow $\ell^1$-rates.
\section{Adaptivity}\label{vtv.s.3}

In this section we first expose some notation for our main result. After having exposed our main result, Theorem \ref{vtv.t.3.1}, we work out explicit expressions for the bound on the antiprojections ${\tilde v}$, the inverse scaling factor ${\tilde \gamma}$ and the noise weights $v$. Finally, we show a bound on the effective sparsity via a suitable interpolating tensor. In Subsection \ref{vtv.ss.3.5} we put the pieces together to prove Theorem \ref{vtv.t.3.1}.

Fix $ k \in [1:\min_{i \in [d]}n_i-1]$ and an active set $S \subseteq \ttimes{i \in [d]} [k+2:n_i-k]$.

To every jump location in $S$, we associate a hyperrectangle of $k^d$ additional jump locations to obtain the enlarged active set $\tS$, defined as
$$ \tS:= \bigcup_{m=1}^s (\ttimes{i \in [d]} [t_{i,m}:t_{i,m}+k-1]).$$

\begin{definition}[Hyperrectangular tessellation]\label{vtv.d.3.1} 
We call $\{ R_m \}_{m=1}^s $ a hyperrectangular tessellation of $\ttimes{i \in [d]} [k+1: n_i ]$
if it satisfies the following conditions:\\
$\bullet$ each $R_m \subseteq \ttimes{i \in [d]} [k+1: n_i ] $ is a hyperrectangle ($m=1 , \ldots , s $);\\
$\bullet$ $ \cup_{m=1}^s R_m = \ttimes{i \in [d]} [k+1: n_i ] $;\\
$\bullet$ for all $m $ and $m^{\prime}\not= m $, the hyperrectangles $R_{m}$ and $R_{m^{\prime}} $ possibly share boundary points but not interior points;\\
$\bullet$ for all $m$, the points $\ttimes{i \in [d]} [t_{i,m}:t_{i,m}+k-1]$ are interior points of $R_m $.
\end{definition}

For a hyperrectangular tessellation $\{ R_m \}_{m=1}^s$ denote the vertices of the hyperrectangle $R_m$ by $(t_{1,m}^{z_1}, \ldots, t_{d,m}^{z_d}), (z_1, \ldots, z_d) \in \{-,+\}^d$ , for $m\in [s]$.

Moreover we define the distances of the jump locations from the vertices of their respective hyperrectangle and the respective set of indices  as
\begin{align*}
 d_{i,m}^- &:= (t_{i,m} - t_{i,m}^- ) ,&R_{i,m}^- &:= [  t_{i,m}^-: t_{i,m} ],\\
 d_{i,m}^0&:= k ,&R_{i,m}^0 &:=  [t_{i,m}:t_{i,m}+k-1],\\
  d_{i,m}^+ & := (t_{i,m}^+ - t_{i,m}-k+1 ),& R_{i,m}^+ & := [ t_{i,m}+k-1: t_{i,m}^+ ],
\end{align*}
for $i \in [d]$ and $m \in [s]$. 
Each hyperrectangle $R_m$ of the hyperrectangular tessellation $\{R_m\}_{m\in [s]}$ can be partitioned into $3^d$ hyperrectangles. 
Define, for all $(z_1, \ldots, z_d) \in \{-,0,+\}^d$,
$$ R_m^{z_1 \cdots z_d}:= R_{1,m}^{z_1} \times \ldots \times R_{d,m}^{z_d}, \ m \in [s] .$$
For $m \in [s]$, let
$$ d_m^{z_1\cdots z_d}:=d_{1,m}^{z_1}\cdot \ldots \cdot d_{d,m}^{z_d}, \ \{z_1, \ldots, z_d\}\in \{-,+\}^d.$$

We define the maximal  distance from  an (enlarged) jump location to the boundary of the corresponding rectangular region along the coordinate $i \in [d]$ as
$$ d_{i, {\rm max} } (S) := \max_{m \in [1:s] } \max\{ d_{i,m}^-, d_{i,m}^+ \}.$$

\begin{figure}[h]
\center
\begin{tikzpicture}
\draw (0,0) rectangle (10/2,10/2);

\draw (0,5/2) -- (6/2,5/2);
\draw (7.5/2,5/2) -- (10/2,5/2);

\draw (6/2,0) -- (6/2,5/2);
\draw (6/2,6.5/2) -- (6/2,10/2);

\draw (0,6.5/2) -- (6/2,6.5/2);
\draw (7.5/2,6.5/2) -- (10/2,6.5/2);

\draw (7.5/2,0) -- (7.5/2,5/2);
\draw (7.5/2,6.5/2) -- (7.5/2,10/2);

\foreach \x in {6/2,6.5/2,7/2,7.5/2} {
        \foreach \y in {5/2,5.5/2,6/2,6.5/2} {
            \node[circle,inner sep=2pt,fill=black] at (\x,\y){};
        }
    }

\draw[pattern=north west lines] (0,0) rectangle (6/2,5/2);
\draw[pattern=north west lines] (7.5/2,6.5/2) rectangle (10/2,10/2);
\draw[pattern=north east lines] (0,6.5/2) rectangle (6/2,10/2);
\draw[pattern=north east lines] (7.5/2,0) rectangle (10/2,5/2);

\node[fill=white,rounded corners=2pt] at (2.8/2,2.4/2) {$d_m^{+,-}$};
\node[fill=white,rounded corners=2pt] at (2.8/2,7.8/2) {$d_m^{-,-}$};
\node[fill=white,rounded corners=2pt] at (8.6/2,7.8/2) {$d_m^{-,+}$};
\node[fill=white,rounded corners=2pt] at (8.6/2,2.4/2) {$d_m^{+,+}$};


\node[circle,inner sep=2pt,fill=gray,label={[text=gray,shift={(-.95,-.75)}]$(t_{1,m}^+,t_{2,m}^-)$}] at (0,0) {};
\node[circle,inner sep=2pt,fill=gray,label={[text=gray,shift={(1.1,-.2)}]$(t_{1,m}^+,t_{2,m}^+)$}] at (10/2,0) {};
\node[circle,inner sep=2pt,fill=gray,label={[text=gray,shift={(0.9,-.2)}]$(t_{1,m}^-,t_{2,m}^+)$}] at (10/2,10/2) {};
\node[circle,inner sep=2pt,fill=gray,label={[text=gray,shift={(-.9,-.2)}]$(t_{1,m}^-,t_{2,m}^-)$}] at (0,10/2) {};

\node[circle,inner sep=2pt,fill=gray,label={[text=gray,shift={(-1.25,-.2)},fill=white]$(t_{1,m},t_{2,m}^-)$}] at (0,6.5/2) {};
\node[circle,inner sep=2pt,fill=gray,label={[text=gray,shift={(-1.9,-.7)},fill=white]$(t_{1,m}+k-1,t_{2,m}^-)$}] at (0,5/2) {};
\node[circle,inner sep=2pt,fill=gray,label={[text=gray,shift={(-.8,-.85)}]$(t_{1,m}^-,t_{2,m})$}] at (6/2,0) {};
\node[circle,inner sep=2pt,fill=gray,label={[text=gray,shift={(+1.4,-.85)}]$(t_{1,m}^+,t_{2,m}+k-1)$}] at (7.5/2,0) {};

\draw [decorate,decoration={brace,amplitude=10pt},xshift=-4pt,yshift=0pt] (0,0.05) -- (0,5/2-.05) node [midway,xshift=-0.8cm] { $d_{1,m}^+$};
\draw [decorate,decoration={brace,amplitude=10pt},xshift=-4pt,yshift=0pt] (0,6.5/2+.05) -- (0,10/2-.05) node [midway,xshift=-0.8cm] { $d_{1,m}^-$};
\draw [decorate,decoration={brace,amplitude=10pt,mirror},yshift=-4pt,xshift=0pt] (0+.05,-.5) -- (6/2-.05,-.5) node [midway,yshift=-0.8cm] { $d_{2,m}^-$};
\draw [decorate,decoration={brace,amplitude=10pt,mirror},yshift=-4pt,xshift=0pt] (7.5/2+.05,-.5) -- (10/2-.05,-.5) node [midway,yshift=-0.8cm] { $d_{2,m}^+$};
\end{tikzpicture}

\caption{A rectangle of the tessellation $\{R_m\}_{m=1}^s$ for $d=2$ and $k=4$}\label{vtv.f.3.1}

\end{figure}
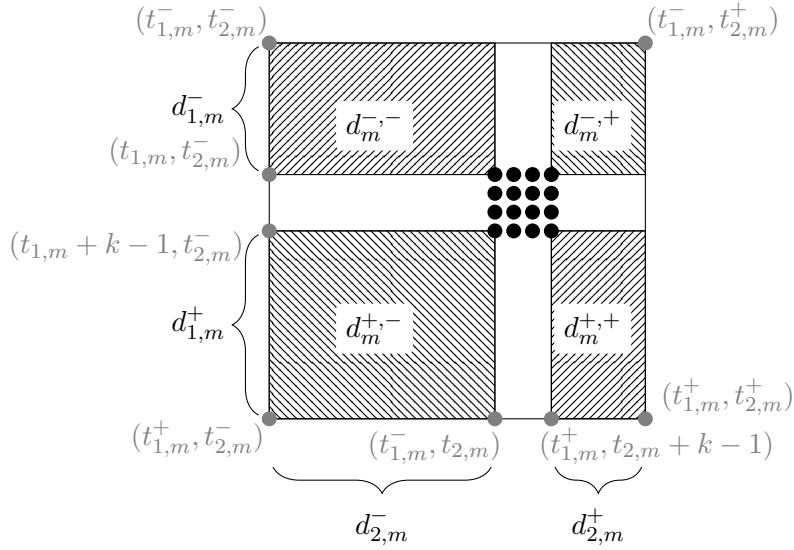

For $d=2$ and $k=4$, a rectangle of the tessellation is depicted in Figure \ref{vtv.f.3.1}.

\subsection{Main result}

We present our main result, that shows that trend filtering leads to an adaptive estimation of the $d$-dimensional margin $f^0_{\N^{\perp}_k}$ of $f^0$.

\begin{theorem}[Adaptivity of trend filtering]\label{vtv.t.3.1}
Fix $k\in\{1,2,3,4\} $ and choose $x,t>0$. Let $g \in \Rd$ be arbitrary. Let $S$ be an arbitrary subset of size $s:=\abs{S}$ of $\ttimes{i \in [d]}[k+1+(k+2)k:n_i-k+1-(k+2)k]$.
For a large enough constant $C>0$ that only dependends on $k$, choose
$$\lambda \ge C \ d \ \sqrt{\sum_{i=1}^d \left(\frac{d_{i,\max}(S)}{n_i} \right)^{2k-1}}\  \lambda_0(t) .$$
Then, with probability at least $1-e^{-x}-e^{-t}$, it holds that
$$\norm{(\hf-f^0)_{\N^{\perp}_k}}^2_2/n \le \norm{g-f^0_{\N^{\perp}_k}}^2_2/n + 4 \lambda \norm{(D^k g)_{-S}}_1+ \frac{2\sigma^2}{n} \left(\sqrt{x}+ \sqrt{ks} \right)^2$$
$$+ \bigo\left(\lambda^2 \left(\sum_{i=1}^d \log(e d_{i,\max}(S))\right) \sum_{m=1}^s \sum_{z \in \{-,+\}^d} \left(\frac{n}{d_m^z} \right)^{2k-1} \right).$$
In particular the constraint on $C$ is
$$ C \ge \frac{k^{\frac{2k-1}{2}}}{a_0} \text{ with } a_0= \begin{cases} 1, & k=1, \\ 8\sqrt{2}/7\approx 1.62, & k=2, \\144\sqrt{3}/76\approx 3.28 , & k=3, \\ 10.10, & k=4, \end{cases}$$
as $\min_{i \in [d]} \min_{m \in [s]}  \min \{d_{i,m}^-, d_{i,m}^+ \} \to \infty$.
\end{theorem}

\begin{proof}
See Subsection \ref{vtv.ss.3.5}.
\end{proof}

By choosing $x \asymp t \asymp \log n$ in Theorem \ref{vtv.t.3.1} and by constraining  the active set $S$ to be a regular grid we retrieve  Theorem \ref{vtv.t.2.1}. In that case, since $S$ is a regular grid, we can choose $\lambda \asymp s^{-\frac{2k-1}{2d}} \lambda_0(\log(2n))$ and the oracle inequality has the rate
$$ \bigo\left( \frac{s^{\frac{2k(d-1)+1}{d}}}{n} \log(n/s) \log n \right) .$$

\begin{remark}[The role of the hyperrectangular tessellation]
Given an active set $S$, the choice of a hyperrectangular tessellation in Theorem \ref{vtv.t.3.1} can be seen as arbitrary.
\end{remark}

\subsection{Some definitions}

We introduce some quantities on which Theorem \ref{vtv.t.3.1} relies: the bound on the antiprojections ${\tilde v}$, the inverse scaling factor ${\tilde \gamma}$, the noise weights $v$, a sign configuration $q$ and the effective sparsity $\Gamma^2_{D^k}$.

Let  $\tS$  be the enlarged active set induced by some active set $S$. Let ${\rm P}_{\tS}$ be the orthogonal projection operator on ${\rm span}(\{\tf^k_{j_1, \ldots, j_d}\}_{(j_1, \ldots, j_d)\in \tS} )$.

\begin{definition}[Bound on the antiprojections]\label{vtv.d.2.7}
A bound on the antiprojections is a tensor $\tilde{v} \in \R^{(n_1-k)\times \ldots \times (n_d-k)}$ such that
$${\tilde v}_{j_1, \ldots, j_d} \ge \norm{({\rm I}-{\rm P}_{\tS}) {\tilde \phi}^k_{j_1, \ldots, j_d}}_2/ \sqrt{n}, \ \forall (j_1, \ldots, j_d) \in \ttimes{i \in [d]} [k+1:n_i].$$
\end{definition}

Let $\tilde v$ be a bound on the antiprojections.

\begin{definition}[Inverse scaling factor]\label{vtv.d.2.8}
The inverse scaling factor $\tilde{\gamma} \in \R$ is defined as
$\tilde{\gamma}:= \norm{\tilde{v}_{-\tS}}_{\infty}$.
\end{definition}

Let $\tilde v$ be a bound on the antiprojections and $\tilde \gamma$ the corresponding inverse scaling factor.

\begin{definition}[Noise weights]\label{vtv.d.9}
The noise weights $v\in \R^{(n_1-k)\times \ldots \times (n_d-k)}$ are defined as
 $v\ge  {\tilde v}/ {\tilde \gamma} \in [0,1]^{(n_1-k)\times \ldots \times (n_d-k)}$.
\end{definition}

We can now introduce the effective sparsity. The effective sparsity depends on a so-called ``sign configuration'', that is, on the sign pattern associated with the jump locations.

\begin{definition}[Sign configuration]\label{vtv.d.10}
Let $q \in [-1,1]^{(n_1-k)\times \ldots \times (n_d-k)}$ be s.t.
$$q_{j_1, \ldots, j_d}\in\begin{cases} \{-1,+1\},&  (j_1, \ldots, j_d)=t_m\in S, \\
\{q_{t_m}\}, & (j_1, \ldots, j_d) \in \ttimes{i \in [d]}[t_{i,m}:t_{i,m}+k-1],\ m\in [s],\\
[-1,1], & (j_1, \ldots, j_d)\notin \tS.\end{cases}$$
We call  $q_S \in \{-1,0,1\}^{(n_1-k)\times \ldots \times (n_d-k)}$ a sign configuration.
\end{definition}

The basic definition of effective sparsity depends on the sign configuration associated with $S$. One can however  remove this dependence by defining the effective sparsity as the maximum over all sign configurations.

\begin{definition}[Effective sparsity]\label{vtv.d.2.11}
Let an active set $S $, a sign configuration $q_S$  and noise weights $v$ be given. The effective sparsity $\Gamma^2_{D^k} ( S, {v}_{-S},q_S )\in \R$ is defined as
$$\Gamma_{D^k} ( S, {v}_{-S},q_S ) :=$$
$$:=\max  \left\{   \sum_{m=1}^s (q_S)_{t_m}  (D^k f)_{t_m}  - \| (1- {v})_{-S}  \odot ( D^k f)_{-S} \|_1  :\| f \|_2^2 / n=1 \right\}.$$
Moreover we write 
$$ \Gamma^2_{D^k} ( S, {v}_{-S}):= \max_{q_S}\ \Gamma^2_{D^k} ( S, {v}_{-S},q_S ).$$
\end{definition}

By the adaptive bound of Theorem 2.2 in \citet{vand19} (see also Theorem 2.1 in \citet{orte19-2} and Theorem 16 in \citet{orte19-4} modified with an enlarged active set), we know that bounding the effective sparsity is a sufficient condition for proving adaptation of $\hf_{\N^{\perp}_k}$.

\subsection{Effective sparsity via interpolating tensors}

To bound the effective sparsity we extend the technique by \citet{vand19} involving interpolating vectors to interpolating tensors, i.e., tensors that interpolates the signs of the jumps.

\begin{definition}[Interpolating tensor]\label{vtv.d.2.12}
Let $q_S \in \{ -1,0,1 \}^{(n_1-k)\times \ldots \times (n_d-k)} $ be a sign configuration and $v \in [0,1]^{(n_1-k)\times \ldots \times (n_d-k)}$ be a tensor of noise weights.
The tensor $w (q_S)\in \R^{(n_1-k)\times \ldots \times (n_d-k)}$ is called  an interpolating tensor for the sign configuration $q_S$ and the weights $v$ if it has the following properties:\\
$\bullet$ $w_{j_1, \ldots, j_d}(q_S) = (q_S)_{t_m} $, $ \forall (j_1, \ldots, j_d) \in \ttimes{i\in [d]} [t_{i,m}: t_{i,m}+k-1] $, $\forall m\in [s] $,\\
$\bullet$ $|w_{j_1,\ldots,j_d}(q_S)  | \le 1- {v}_{j_1,\ldots,j_d}  , \ \forall (j_1,\ldots,j_d)\in (\ttimes{i \in [d]} [ k : n_i]) \setminus \tS $.
\end{definition}

With the help of an interpolating tensor we can bound the effective sparsity, as the following lemma shows (Lemma 2.4 by \citet{vand19} in tensor form).

\begin{lemma}[Bounding the effective sparsity with an interpolating tensor]\label{vtv.l.2.2}
We have
$$\Gamma^2_{D^k}  (S, {v}_{-S},q_S ) \le n \min_{w(q_S)}  \|(D^k)' w(q_S)\|_2^2$$
where the minimum is over all interpolating tensors $w(q_S) $ for the sign configuration $q_S$.  
\end{lemma}

\begin{proof}
It holds that
\begin{eqnarray*}
&&\sum_{m=1}^s (q_S)_{t_m} (D^k f)_{t_m} - \norm{(1-v)_{-S} \odot (D^k f)_{-S}}_1\\
& \le & \sum_{m=1}^s (q_S)_{t_m} (D^k f)_{t_m} - \norm{w(q_S)_{-S} \odot (D^k f)_{-S}}_1\\
& \le & \sum_{2, \ldots, 2}^{n_1, \ldots, n_d} w(q_S)_{j_1, \ldots, j_d} (D^k f)_{j_1, \ldots, j_d}\\
& = & \sum_{1, \ldots, 1}^{n_1, \ldots, n_d} ((D^k)' w(q_S))_{j_1, \ldots, j_d}  f_{j_1, \ldots, j_d}\\
& \le & \sqrt{n} \norm{(D^k)' w(q_S)}_2 \norm{f}_2/\sqrt{n}.
\end{eqnarray*}
\end{proof}

\subsection{Requirements on an interpolating tensor}

Theorem \ref{vtv.t.3.1} follows by a bound on the effective sparsity obtained by Lemma \ref{vtv.l.2.2} with the help of an interpolating tensor.
In the definition of an interpolating tensor (cf. Definition \ref{vtv.d.2.12}), there is a constraint posed by the noise weights $v$.

Therefore, we now calculate in Subsection \ref{vtv.ss.3.3} a bound on the antiprojections ${\tilde v}$ to derive an appropriate inverse scaling factor ${\tilde \gamma}$ and noise weights ${v}$. In this way we will make explicit the constraints that the interpolating tensor has to satisfy in the specific case of tensor denoising with trend filtering.

After that, we will show in Subsection \ref{vtv.ss.3.4} an explicit form for the interpolating tensor for $k=\{ 1,2,3,4\}$ and derive the corresponding bound on the effective sparsity.

That bound on the effective sparsity combined with the fact that the interpolating tensor used indeed is an interpolating tensor for trend filtering will allow us to derive Theorem \ref{vtv.t.3.1} from Theorem \ref{vtv.t.2.3}.

\subsection{Antiprojections, inverse scaling factor and noise weights}\label{vtv.ss.3.3}

We start by finding a bound on the antiprojections $\tv$.

Define, for $m \in [s]$ and $i \in [d]$,
$$ {\tilde v}_{i,m}^2(j_i)= \begin{dcases} \left(\frac{t_{i,m}-j_i}{n_i}  \right)^{2k-1}, & j_i \in R_{i,m}^-= [t_{i,m}^-:t_{i,m}],\\
 0, & j_i \in R_{i,m}^0=[t_{i,m}:t_{i,m}+k-1],\\
\left(\frac{j_i-t_{i,m}-k+1}{n_i}  \right)^{2k-1}, & j_i \in R_{i,m}^+=[t_{i,m}+k-1:t_{i,m}^+].\\
\end{dcases}$$

Moreover, for $(j_1, \ldots, j_d) \in R_m$ we define 
$$ {\tilde v}_{j_i, \ldots, j_d}:= \sqrt{ \sum_{i=1}^d {\tilde v}^2_{i,m}(j_i)}.$$

\begin{lemma}[A valid bound on the antiprojections]\label{vtv.l.3.3}
For all $(j_1, \ldots, j_d) \in R_m$ and for all $m\in [s]$ it holds that
$$ \norm{{\rm A}_{\tS} \tf^k_{j_1, \ldots, j_d}}^2_2/n \le {\tilde v}^2_{j_i, \ldots, j_d},$$
i.e., the tensor ${\tilde v}\in \R^{(n_1-k) \times \ldots \times (n_d-k)}$ is a valid bound on the antiprojections.
\end{lemma}

\begin{proof}
See Appendix \ref{vtv.pl.3.3}.
\end{proof}

Define, for $m \in [s]$ and $i \in [d]$,
$$ { v}_{i,m}^2(j_i)= \begin{dcases} \left(\frac{t_{i,m}-j_i}{d_{i,m}^-}  \right)^{2k-1}, & j_i \in R_{i,m}^-=[t_{i,m}^-:t_{i,m}],\\
 0, & j_i \in R_{i,m}^0=[t_{i,m}:t_{i,m}+k-1],\\
\left(\frac{j_i-t_{i,m}-k+1}{d_{i,m}^+}  \right)^{2k-1}, & j_i \in R_{i,m}^+=[t_{i,m}+k-1:t_{i,m}^+].\\
\end{dcases}$$

Moreover, for a constant $C=C(k)\ge 1$ we define  for $(j_1, \ldots, j_d) \in R_m$ and $m=1, \ldots, s$
\begin{equation}\label{vtv.e.3.1}
{ v}_{j_i, \ldots, j_d}:=\frac{1}{d} { \sum_{i=1}^d \frac{{ v}_{i,m}(j_i)}{C}}
\end{equation}
and
$$ {\tilde \gamma}= Cd \sqrt{\sum_{i=1}^d \left(\frac{d_{i,\max}(S)}{n_i} \right)^{2k-1}}.$$

\begin{lemma}[Valid noise weights] \label{vtv.l.3.4} 
For all $m \in [s]$ and for all $(j_1, \ldots, j_d) \in R_m$ it holds that
$${ \tilde v}_{j_i, \ldots, j_d} \le { v}_{j_i, \ldots, j_d} {\tilde \gamma},$$
i.e., the tensor $v \in \R^{(n_1-k) \times \ldots \times (n_d-k)}$ as defined in Equation \eqref{vtv.e.3.1} defines valid noise weights.
\end{lemma} 

\begin{proof}
See Appendix \ref{vtv.pl.3.4}.
\end{proof}

The constant $C\ge 1$ in Equation \eqref{vtv.e.3.1} can be chosen aribtrarily. Choosing a larger $C$ makes the noise weights smaller. As a result,  the requirements imposed on the interpolating tensor by the noise weights become weaker.

\subsection{Bound on the effective sparsity for trend filtering}\label{vtv.ss.3.4}

We now define an interpolating tensor $w=w(q)$ for any sign configuration $q_S$.

For  $(j_1, \ldots, j_d) \in R_m, m \in [s]$ and the same constant $C=C(k)>0$ as in the definition of the noise weights in Equation \eqref{vtv.e.3.1}, define the tensor
\begin{equation}\label{vtv.e.3.2}
w_{j_1, \ldots, j_d}(q_S):= \frac{1}{d} \sum_{i=1}^d \prod_{l=1}^d w_{l,i,m}(j_l),
\end{equation}
where,
\begin{align*}
w_{l,i,m}(j_l)&= q_{t_m},\  &j_l &\in R_{l,m}^0,& l\not=i,\\
 w_{l,i,m}(j_l)&\in [0,q_{t_m}], &j_l &\in R_{l,m}^- \cup R_{l,m}^+,& l\not=i,\\
w_{i,i,m}(j_i)&=q_{t_m},&  j_i &\in R_{i,m}^0,& l  =i,\\
w_{i,i,m}(j_i)&\le  q_{t_m} (1- {v_{i,m}(j_i)}/{C}),&  j_i &\in R_{l,m}^- \cup R_{l,m}^+,&  l  =i.
\end{align*}
What differentiates the case $l=i$ is that  $w_{i,i,m}$ has to satisfy the requirements imposed by the noise weights. For $l \not= i$ the only constraint imposed is that $\abs{w_{l,i,m}}\le 1$.
The tensor $w$ is a sum of terms with product structure if constrained to the set of indices of a hyperrectangle $R_m$.

We define $w_{l,i,m}^-:=\{w_{l,i,m}(j_l)\}_{j_l \in R^-_{i,m}}$ and $w_{l,i,m}^+:=\{w_{l,i,m}(j_l)\}_{j_l \in R^+_{i,m}}$.

\begin{lemma}[A valid interpolating tensor]\label{vtv.l.3.5}
For any given sign configuration $q_S$, the tensor $w=w(q_S)$ defined in Equation \eqref{vtv.e.3.2} is a valid interpolating tensor.
\end{lemma}

\begin{proof}
See Appendix \ref{vtv.pl.3.5}.
\end{proof}

\subsubsection{Matching derivatives}\label{vtv.sss.3.4.1}

We now want to find the explicit form of an appropriate interpolating tensor $w$, to apply in Lemma \ref{vtv.l.2.2}. We first consider continuous versions $\omega(x)$, respectively. $\rw(x)$, of the  vectors $w_{i,i,m}^{-}$ and $w_{i,i,m}^{+}$, respectively. $w_{l,i,m}^{-}$ and $w_{l,i,m}^{+}$ for $ l \not =i$, on a mock interval $x \in [0,1]$. We then set
\begin{align*}
w_{i,i,m}^-(j_i)&:= \omega\left(\frac{t_{i,m}-j_i}{d_{i,m}^-} \right), & j_i \in R^-_{i,m},\\
w_{i,i,m}^+(j_i)&:= \omega\left(\frac{j_i-t_{i,m}-k+1}{d_{i,m}^+} \right), & j_i \in R^+_{i,m},\\
w_{l,i,m}^-(j_l)&:= \rw\left(\frac{t_{l,m}-j_l}{d_{l,m}^-} \right), & j_l \in R^-_{l,m},\\
w_{l,i,m}^+(j_l)&:= \rw\left(\frac{j_l-t_{l,m}-k+1}{d_{l,m}^+} \right), & j_l \in R^+_{l,m},
\end{align*}
for $i \in [d], l \not= i, m \in [s]$.

We aim to find a form of $\omega$ and $\rw$ giving place to continuous functions with $k-1$ continuous derivatives and piecewise constant $k^{\rm th}$ derivative. Moreover, these functions have to be interpolating between the jump location ($x=0$) and the border ($x=1$). We guarantee that they interpolate the signs of the jumps by restricting to polynomials with
\begin{align*}
\omega(0)&=1, & \omega(1)&=0, &  \\
\rw(0)&=1, & \rw(1)&=0, & \rw(x)&= 1- \rw(1-x), x \in [0,1].
\end{align*}
The discretized version of these polynomials will vanish at the boundaries of the hyperrectangles while it will have the value $1$ at the indices belonging to the enlarged active set $\tS$, guaranteeing the interpolation of the signs of the jump locations. Moreover, we will have to choose the constant $C>0$ in Equation \eqref{vtv.e.3.1} such that the noise weights are made small enough for the interpolating polynomial to satisfy the conditions of Lemma \ref{vtv.l.3.5}.

To obtain interpolating polynomials $\omega$ and $\rw$, we split the interval $[0,1]$ into an adequate number of subintervals. We then choose $\omega$ and $\rw$ to be made of polynomial pieces of order at most $k$. The exception is the first subinterval $[0,x_1], x_1\in (0,1]$ for $\omega$, where we choose $\omega(x)=1-a_0x^{\frac{2k-1}{2}}$. We then find the explicit values of the coefficients of the polynomials by derivatives matching, as in \citet{vand19}. More details on derivatives matching are given in the Appendix \ref{vtv.match.derivatives}.
 
To guarantee that $\omega$ and $\rw$ can give place to interpolating tensors, one has to check that derivative matching renders a piecewise polynomial which is monotone. Monotonicity combined with the constraints $\omega(0)=\rw(0)=1$ and $\omega(1)=\rw(1)=0$ ensures that $\abs{\omega(x)}\le 1$ and $\abs{\rw(x)}\le 1$.
 
Monotone interpolating polynomials $\omega$ and $\rw$ and a large enough $C$ in the tuning parameter are sufficient conditions for a valid interpolating tensor. In particular, given that $\omega$ is monotone, we require that
\begin{equation}\label{vtv.e.3.4}
C \ge k^{\frac{2k-1}{2}}/a_0 \text{ as } \min_{i \in [d]} \min_{m \in [s]}  \min \{d_{i,m}^-, d_{i,m}^+ \} \to \infty.
\end{equation}
Note that for the construction of $\rw$, we do not have any constraint given by the antiprojections ${\tilde v}$, the noise weights $v$ and the inverse scaling factor ${\tilde \gamma}$. Therefore, we can take the dependence on $x^k$ instead of $x^{\frac{2k-1}{2}}$. This saves a logarithmic term, not visible in Lemma \ref{vtv.l.3.7}, which only contains the logarithmic terms stemming from $\omega$. Indeed, as Lemma \ref{vtv.l.3.6} in Appendix \ref{vtv.pl.3.6} shows, partial integration of a $k^{\rm th}$-order polynomial does not incur in log terms, while partial integration of $x^{\frac{2k-1}{2}}$ does so. We have to choose the worse dependence on $x^{\frac{2k-1}{2}}$ for $\omega$ though, because  $\omega$ has to respect the constraints posed by the noise weights.

\subsubsection{Show a bound on the effective sparsity}

We now show a bound on the effective sparsity, using a ``candidate'' interpolating tensor generated from the discretizations of $\omega$ and $\rw$ whose construction has been exposed above. 
We call it ``candidate'' interpolating tensor because we have not yet shown that $\omega$ and $\rw$ are monotone.
For the moment we assume that matching derivatives renders monotone $\omega$ and $\rw$. We check the monotonicity for $k=\{1,2,3,4\}$ in the next subsection.

To make the notation and the computation steps lighter, we neglect the constants and use the order notation $\bigo$ instead.

Since the sign configuration $q_S$ is typically unknown, we focus on finding an upper bound on the effective sparsity that does not depend on the sign configuration $q_S$. Thus, the bound also accommodates for the worst-case sign configuration.

\begin{lemma}[Effective sparsity for trend filtering]\label{vtv.l.3.7}
Take the interpolating vector $w$ as defined in Equation \eqref{vtv.e.3.2}. Choose the  vectors $w_{i,i,m}^{-}$ and $w_{i,i,m}^{+}$, respectively $w_{l,i,m}^{-}$ and $w_{l,i,m}^{+}$ for $ l \not =i$, to be discretized versions of  $\omega(x)$ and $\rw(x)$ as in Subsubsection \ref{vtv.sss.3.4.1}.
Assume that $\omega(x)$ and $\rw(x)$ obtained by derivative matching are monotone.

For such an interpolating vector $w$, it holds that
\begin{equation*} 
\Gamma^2_D(S,v_{-S}) = \bigo \left( \left(\sum_{i=1}^d \log( e d_{i, \max}(S))\right) \sum_{m=1}^s  \sum_{z \in \{-,+\}^d} \left(\frac{n}{d_m^z}\right)^{2k-1} \right)
\end{equation*}
\end{lemma}

\begin{proof}
See Appendix \ref{vtv.pl.3.7}.
\end{proof}

From Lemma \ref{vtv.l.3.6} and the matching of discrete derivatives, it follows that, if $\omega$ and $\rw$ are monotone and $C$ is chosen large enough
$$ \Gamma^2_D (S, v_{-S})= \bigo\left(\left(\sum_{i=1}^d \log(e d_{i,\max}(S))\right) \sum_{m=1}^s \sum_{z \in \{-,+\}^d} \left(\frac{n}{d_m^z} \right)^{2k-1} \right).$$
If the active set $S$ defines a regular grid we therefore have a bound on the effective sparsity of order
$$ \Gamma^2_D (S, v_{-S})= \bigo\left(s^{2k} \log(n/s) \right).$$

It only remains to check the monotonicity of $\omega$ and $\rw$. We will do this for $k=\{1, 2, 3, 4\}$. One can  check monotonicity for higher values of $k$ by solving (for instance at the computer) the appropriate system of equations and, say, graphically visualizing the result. We check monotonicity analytically for $k=\{1,2,3\}$ and computationally for $k=4$.

\subsubsection{Interpolating tensor for $k=1$}\label{vtv.sss.3.4.3}

For $k=1$
$$ \omega(x)= 1- \sqrt{x},\  x \in [0,1]. $$
and 
$$ \rw (x)= 1-x,\ x \in [0,1].$$
Both $\omega$ and $\rw$ are monotone.

\subsubsection{Interpolating tensor for $k=2$}\label{vtv.sss.3.4.4}

For $k=2$ 
$$ \omega(x)= \begin{dcases}
1- \frac{8\sqrt{2}}{7} x^{3/2},&  x \in [0,1/2],\\
\frac{12}{7}(1-x)^2,&  x \in [1/2,1],\end{dcases}
$$
and 
$$ \rw(x)= \begin{dcases}
1- \frac{8}{3} x^{2}, & x \in [0,1/4],\\
\frac{4}{3} \left(\frac{1}{2}-x\right)+ \frac{1}{2}, & x \in [1/4,1/2].\end{dcases}
$$
Both $\omega$ and $\rw$ are monotone.

\subsubsection{Interpolating tensor for $k=3$}\label{vtv.sss.3.4.5}

For $k=3$ 
$$ \omega(x)= \begin{dcases}
1- \frac{144\sqrt{3}}{76} x^{5/2},&  x \in [0,1/3],\\
\frac{585}{76}x^3-\frac{45}{4}x^2+\frac{255}{76}x+\frac{145}{228} ,&  x \in [1/3,2/3],\\
\frac{315}{76}(1-x)^3,&  x \in [2/3,1],
\end{dcases}
$$
and 
$$ \rw(x)= \begin{dcases}
1- \frac{16}{3} x^{3}, & x \in [0,1/4],\\
-\frac{16}{3} \left(\frac{1}{2}-x\right)^3+2 \left(\frac{1}{2}-x\right)+ \frac{1}{2}, & x \in [1/4,1/2].\end{dcases}
$$
Both $\omega$ and $\rw$ are monotone.

\subsubsection{Interpolating tensor for $k=4$}\label{vtv.sss.3.4.6}

For $k=4$ 
$$ \omega(x)= \begin{dcases}
1- 7.29 x^{7/2},&  x \in [0,1/4],\\
27.39 x^4 - 35.36 x^3 + 12.26 x^2 -2.01 x + 1.12, & x \in [1/4, 1/2],\\
-29.51 x^4 + 78.43 x^3 - 73.08 x^2 +26.44 x - 2.43, & x \in [1/2, 3/4],\\
10.10(1-x)^4, & x \in [3/4,1],\end{dcases}
$$
and 
$$ \rw(x)= \begin{dcases}
1- 16.2 x^{4}, & x \in [0,1/6],\\
27 x^4 - 28.8 x^3 + 7.2 x^2 -0.8 x + 1.03, & x \in [1/6, 1/3],\\
-7.2 \left(\frac{1}{2}-x\right)^3+2.2 \left(\frac{1}{2}-x\right)+ \frac{1}{2}, & x \in [1/3,1/2].\end{dcases}
$$
Both $\omega$ and $\rw$ are monotone.

\subsection{Proof of Theorem \ref{vtv.t.3.1}}\label{vtv.ss.3.5}

Theorem \ref{vtv.t.3.1} follows by combining Theorem \ref{vtv.t.2.3} with a bound on the effective sparsity.

Lemma \ref{vtv.l.3.7} uses Lemma \ref{vtv.l.2.2} to give us a bound on the effective sparsity holding for all sign configurations. This bound is based on a specific form of the interpolating tensor, obtained by derivative matching as explained in Subsection \ref{vtv.sss.3.4.1}. The interpolating tensor obtained by derivative matching is valid if the monotonicity of $\omega$ and $\rw$ is guaranteed. In Subsections \ref{vtv.sss.3.4.3}-\ref{vtv.sss.3.4.6} we check that the interpolating tensors obtained by derivative matching for $k=\{1,2,3,4\}$ satisfy the monotonicity requirement.

There is also a constraint posed by the noise weights and by the constant $C$ to satisfy.
Lemma \ref{vtv.l.3.3} gives a valid bound on the antiprojections. If we choose 
$$ {\tilde \gamma}= Cd \sqrt{\sum_{i=1}^d \left(\frac{d_{i,\max}(S)}{n_i} \right)^{2k-1}}$$
the noise weights given in Equation \eqref{vtv.e.3.1} are valid noise weights, according to Lemma \ref{vtv.l.3.4}. By Lemma \ref{vtv.l.3.5}, interpolating tensors of the form given in Equation \eqref{vtv.e.3.2} are valid interpolating tensors. The  tensor obtained by the discretization of the result of derivative matching has such a form (as $\min_{m \in [s]} \min_{i \in [d]} \min \{d_{i,m}^-, d_{i,m}^+ \} \to \infty$).

According to Equation \eqref{vtv.e.3.4} in Subsubsection \ref{vtv.sss.3.4.1} one has to choose
$$ C \ge k^{\frac{2k-1}{2}}/a_0 \text{ as }  \min_{i \in [d]} \min_{m \in [s]} \min \{d_{i,m}^-, d_{i,m}^+ \} \to \infty.$$
The values of $a_0$ are given in Subsubsections \ref{vtv.sss.3.4.3}-\ref{vtv.sss.3.4.6}.

Theorem 2.2 by \citet{vand19}, on which Theorem \ref{vtv.t.2.3} is based, uses a bound on the increments of empirical process $\{\epsilon'f, f \in \R^n\}$, where $\epsilon$ has i.i.d. entries. Theorem \ref{vtv.t.2.3}  involves in the background an empirical process, whose increments are given by
$$ \left\{\sum_{1, \ldots, 1}^{n_1, \ldots, n_d} ( \epsilon_{\N^{\perp}_k} \odot f)_{j_1, \ldots, j_d}, f \in \R^{n_1 \times \ldots \times n_d} \right\}$$
Note that the entries of $\epsilon_{\N^{\perp}_k}={\rm P}_{\N^{\perp}_k} \epsilon$ are correlated. However, by the idempotence of orthogonal projections, we can work with uncorrelated errors and instead restrict to tensors  $f_{\N^{\perp}_k} \in \N^{\perp}_k$. Indeed $\sum_{1, \ldots, 1}^{n_1, \ldots, n_d} ( \epsilon_{\N^{\perp}_k} \odot f)_{j_1, \ldots, j_d}=\sum_{1, \ldots, 1}^{n_1, \ldots, n_d} ( \epsilon \odot f_{\N^{\perp}_k})_{j_1, \ldots, j_d}$. This allows us to take over the arguments of Theorem 2.2 by \citet{vand19}.

\hfill $\square$

\begin{remark}[The influence of the dimensionality]
If we choose $\lambda \asymp \tilde{\gamma} \lambda_0(t)$, the rate of the oracle inequality is ${\tilde \gamma}^2 \sum_{m=1}^s \sum_{z \in \{-,+\}^d} ({n}/{d_m^z})^{2k-1}/n$, up to logarithmic factors. For simplicity, let $S$ define a regular grid. Then the (hyper-)volume of one of the $s$ hyperrectangles of the tessellation scales as $d_m^z\asymp n/s$. Hence the scaling $\sum_{m=1}^s \sum_{z \in \{-,+\}^d} ({n}/{d_m^z})^{2k-1}\asymp s^{2k}$. However ${\tilde \gamma}$, the maximal length of an antiprojection, scales as ${\tilde \gamma} \asymp (s^{-\frac{1}{d}})^{\frac{2k-1}{2}}$, where $s^{-\frac{1}{d}}\asymp d_{i, \max}/n_i$ is proportional to the side length of a hyperrectangle of the tessellation, up to the exponent $2k-1$. The influence of the dimensionality in the exponent of $s$ is a consequence of the different scaling of volume and side length of a hyperrectangle in $d$-dimensions. The (hyper-)volume scales as $s^{-1}$ while the side length scales as $s^{-\frac{1}{d}}$. The reason for this discrepancy is that we are not able to find an upper bound for the noise weights proportional the volume of the hyperrectangles, i.e., to the product of side lengths. The bound we obtain involves rather the sum of side lengths.
\end{remark}

\section{Not-so-slow $\ell^1$-rates}\label{vtv.s.4}

Theorem \ref{vtv.t.2.2} about not-so-slow rates for trend filtering is based on Theorem \ref{vtv.t.2.4}, where the choice of the active set $S$ is aribtrary. The criterion guiding choice of $S$ is to get an ``as small as possible'' value of the inverse scaling factor ${\tilde \gamma}$. Recall that the inverse scaling factor ${\tilde \gamma}$ is the maximal length of the antiprojection of a dictionary atom $\tf^k_{j_1, \ldots, j_d}$ onto the set of dictionary atoms indexed by the active set $S$, that ist ${\tilde \gamma}\ge \max_{(j_1, \ldots, j_d)\in [n_1]\times \ldots \times [n_d]} \norm{{\rm A}_{S} \tf^k_{j_1, \ldots, j_d}}_2/\sqrt{n}$.

The  active set $S$ could be chosen as a regular grid parallel to the coordinate axes. However, we will show that we can shorten the maximal length of the antiprojections by choosing an active set defining a so-called ``mesh grid'', whose construction we illustrate hereafter.

\subsection{Mesh grids}

Let $\delta \in \mathbb{N}$. For a coordinate $i \in [d]$, we define the set of indices $Z_i(l)$ such that
$$ Z_i(l)=\{ \delta^{\frac{d}{l}} \text{ equispaced indices in } [n_i]\}, l \in [d]$$
and
$$ Z_i(1) \supseteq Z_i(2) \supseteq \ldots \supseteq Z_i(d).$$
If, for any $l \in [d]$, $n_i$ is not a multiple of $\delta^{\frac{d}{l}}$, we relax the requirement on the indices to be approximately equispaced, i.e., the distance between all the indices has to be asymptotically of the same order.
For $i \in [d]$, we also define
$$ {\tilde Z}_i(l)= \bigcup_{h=0}^{k-1} \{Z_i(l)+h\}, l \in [d].$$
Let now $(l_1, \ldots, l_d)\in [d]^d$ be a tuple of indices. We define the set
$$ \calS:= \{(l_1, \ldots, l_d)\in [d]^d: \abs{\{i \in [d]:l_i\le z\}}\le z, \forall z \in [d] \}.$$

\begin{definition}[Mesh grid]\label{vtv.d.4.1}
A mesh grid $S$ is defined as
$$ S:= \bigcup_{(l_1,\ldots,l_d)\in \calS} \left(\ttimes{i \in [d]} Z_i(l_i)  \right).$$
\end{definition}

Figure \ref{vtv.f.4.1} illustrates  a mesh grid for $d=2$.

We now want to enlarge a mesh grid $S$ to allow us to handle $k^{\rm th}$-order trend filtering for $k>1$. 

\begin{definition}[Enlarged mesh grid]\label{vtv.d.4.2}
An enlarged mesh grid $\tS$ is defined as
$$ \tS:= \bigcup_{(l_1,\ldots,l_d)\in \calS} \left(\ttimes{i \in [d]} {\tilde Z}_i(l_i)  \right).$$
\end{definition}

Figure \ref{vtv.f.4.2} illustrates  an enlarged mesh grid for $d=2$ and $k=2$.

Let $s:= \abs{S}$ and ${\tilde s}:= \abs{\tS}$. It holds that $s \asymp {\tilde s}\asymp \prod_{i=1}^d \delta^{\frac{d}{i}}\asymp \delta^{dH(d)}$, where $H(d)= \sum_{i=1}^d 1/i$ is the $d^{\rm th}$ harmonic number. Therefore $\delta \asymp s^{\frac{1}{dH(d)}}$.

\begin{figure}[h]
\centering
  \begin{minipage}[t]{0.49\textwidth}
    \begin{tikzpicture}

    \foreach \x in {0,.1,...,6.1} {
        \foreach \y in {-.1,0,.1,..., 6.1} {
            \fill[color=gray] (\x,\y) circle (0.01);
        }
    }
\foreach \x in {1,3,...,5} {
        \foreach \y in {0,.5, ...,6} {
            \fill[color=black] (\x,\y) circle (0.03);
        }
    }
\foreach \x in {0, .5,...,6} {
        \foreach \y in {1,3,...,5} {
            \fill[color=black] (\x,\y) circle (0.03);
        }
    }
%
\foreach \x in {1,3,...,5} {
        \foreach \y in {1,3,...,5} {
            \fill[color=black] (\x,\y) circle (0.05);
        }
    }

\end{tikzpicture}

\caption{Mesh grid for $d=2$.}\label{vtv.f.4.1}
  \end{minipage}
\begin{minipage}[t]{0.49\textwidth}
\begin{tikzpicture}

    \foreach \x in {0,.1,...,6.2} {
        \foreach \y in {-.1,0,.1,..., 6.1} {
            \fill[color=gray] (\x,\y) circle (0.01);
        }
    }
\foreach \x in {1,3,...,5} {
        \foreach \y in {0,.5, ...,6} {
            \fill[color=black] (\x,\y) circle (0.03);
        }
    }
    
\foreach \x in {1,3,...,5} {
        \foreach \y in {-.1,0.4,.9, ...,5.9} {
            \fill[color=black] (\x,\y) circle (0.03);
        }
    }
    
\foreach \x in {0, .5,...,6} {
        \foreach \y in {1,3,...,5} {
            \fill[color=black] (\x,\y) circle (0.03);
        }
    }
    
\foreach \x in {.1, .6,...,6.1} {
        \foreach \y in {1,3,...,5} {
            \fill[color=black] (\x,\y) circle (0.03);
        }
    }
%
\foreach \x in {1,3,...,5} {
        \foreach \y in {1,3,...,5} {
            \fill[color=black] (\x,\y) circle (0.05);
        }
    }
    
\foreach \x in {1.1,3.1,...,5.1} {
        \foreach \y in {1,3,...,5} {
            \fill[color=black] (\x,\y) circle (0.05);
        }
    }
    
\foreach \x in {1,3,...,5} {
        \foreach \y in {.9,2.9,...,4.9} {
            \fill[color=black] (\x,\y) circle (0.05);
        }
    }
    
\foreach \x in {1.1,3.1,...,5.1} {
        \foreach \y in {.9,2.9,...,4.9} {
            \fill[color=black] (\x,\y) circle (0.05);
        }
    }
    
\end{tikzpicture}

\caption{Enlarged mesh grid for $d=2$ and $k=2$.}\label{vtv.f.4.2}
  \end{minipage}

\end{figure}

\subsection{The inverse scaling factor when $\tS$ is an enlarged mesh grid}

We will now show that we can find a smaller bound on the inverse scaling factor if we choose $\tS$ to be an enlarged mesh grid rather than an enlarged regular grid.

\begin{lemma}[Inverse scaling factor when $\tS$ is an enlarged mesh grid] \label{vtv.l.4.1}
Let $n_1 \asymp \ldots \asymp n_d$ and  $\tS$ be an enlarged mesh grid. It holds that
$$ {\tilde \gamma}(\tS)= \bigo \left(s^{- \frac{2k-1}{2H(d)}} \right).$$
\end{lemma}

\begin{proof}
See Appendix \ref{vtv.pl.4.1}.
\end{proof}


\subsection{Proof of Theorem \ref{vtv.t.2.2}}\label{vtv.pt.2.2}

Theorem \ref{vtv.t.2.2} follows from Theorem \ref{vtv.t.2.4}. Theorem \ref{vtv.t.2.4} is allowed to have correlated errors for the same reasons as Theorem \ref{vtv.t.2.3} is, see the proof of Theorem \ref{vtv.t.3.1} in Subsection \ref{vtv.ss.3.5}.

In Theorem \ref{vtv.t.2.4} we set $x\asymp t \asymp \log n$. We can then choose the free parameters $S$ and $g\in \Rd$ independently of each other. Remember that the normalization of the TV is included in the definition of the analysis operator $D^k$. Therefore it is natural to restrict the choice of $g$ to the class $\{\rf: \1norm{D^k \rf}= \bigo(1)\}$.

We can then choose $S$ to trade off the terms $ {\tilde \gamma} \lambda_0(\log n)\asymp {\tilde \gamma} \log^{1/2}(n)/ n^{1/2}$ and $s/n$. Typically, if we require $S$ to have a regular structure, we obtain ${\tilde \gamma}=\bigo (s^{-h})$ for some $h=h(d,k) \in \R$. The tradeoff is achieved by choosing
$$ s\asymp n^{\frac{1}{2(1+h)}} \log^{\frac{1}{2(1+h)}}n$$
and gives the slow rate
$$ \frac{s}{n} \asymp n^{-1+\frac{1}{2(1+h)}} \log^{\frac{1}{2(1+h)}}n.$$

We choose the active set to be an enlarged mesh grid $\tS$. Then, by Lemma \ref{vtv.l.4.1}, we can choose $ {\tilde \gamma}= \bigo\left(s^{- \frac{2k-1}{2H(d)}} \right)$ and the claim follows.
\hfill $\square$

\begin{remark}[Mesh grids vs. regular grids]
If we choose a regular grid as active set, according to Lemma \ref{vtv.l.3.4} we obtain $ \tilde{\gamma} \asymp s^{-\frac{2k-1}{2d}}$ and a slow rate
$$ n^{- \frac{d+2k-1}{2d+2k-1}} \log^{\frac{d}{2d+2k-1}}(n),$$
which is slower than the rate obtained with an active set defining a mesh grid.
Indeed, for all $d \ge 1$ it holds that $H(d) \le d$.

In both cases, the slow rate for fixed $k$ goes to $n^{-1/2} \log^{1/2}(n)$ as $d \to \infty$. If $d$ is fixed, the slow rates goes to $n^{-1}$ as $k \to \infty$.
\end{remark}

\begin{remark}[Allow $\lambda$ to depend on ${\rm TV}_k$]
In the proof of Theorem \ref{vtv.t.2.2} we can also drop the restriction $g\in \{\rf: \1norm{D^k \rf}= \bigo(1)\}$ and  trade off ${\tilde \gamma} \log^{1/2}(n) {\rm TV}_k(g)/ n^{1/2}$ and $s/n$.
The tradeoff results in the choice
$$\lambda \asymp n^{- \frac{H(d)+2k-1}{2H(d)+2k-1}}\ \log^{\frac{H(d)}{2H(d)+2k-1}}(n)\ {\rm TV}_k(g)^{-\frac{2k-1}{2H(d)+2k-1}} $$
and gives the rate
$$ n^{- \frac{H(d)+2k-1}{2H(d)+2k-1}}\ \log^{\frac{H(d)}{2H(d)+2k-1}}(n)\  {\rm TV}_k(g)^{\frac{2H(d)}{2H(d)+2k-1}} .$$
\end{remark}
\section{Denoising lower-dimensional margins}\label{vtv.s.5}

In the previous sections we have shown how to estimate $f^0_{\N_k^{\perp}}$ by trend filtering and have established fast adaptive $\ell^0$-rates and not-so-slow $\ell^1$-rates. There is still an open question: how to estimate $f^0_{\N_k}$?

If $n_1 \asymp \ldots \asymp n_d$, the dimension of $\N_k$ is of order $n^{1-1/d}$. Estimating $f^0_{\N_k}$ by least squares would result in a rate of order $n^{-1/d}$ and therefore be limiting for $d \ge 2$.

The approach we take is to decompose $\N_k$ into lower dimensional mutually orthogonal linear spaces, the so-called marginal linear spaces, to which we can apply a lower dimensional version of trend filtering.

Let $\P[d]$ denote the power set of $[d]:= \{1, \ldots, d\}$. We consider sets of coordinate indices $M\subseteq [d]$.

The intuition behind the decomposition into margins is to partition the set of tensor indices into $2^d$ subsets as 
$$ [1:k] \cup [k+1:n_1] \times \ldots \times [1:k] \cup [k+1:n_1].$$
For $M \in \P[d]$ define the set of indices
$$ I_M^k= \ttimes{i \in M} [k+1:n_i] \ttimes{i \not \in M} [1:k].$$
We moreover define the linear spaces 
$$\M(M)= {\rm span} \{\tf^k_{k_1, \ldots, k_d}, (k_1, \ldots, k_d) \in I^k_M  \}, \ M \in \P[d].$$

Note that in one dimension, $\{\tf^k_j\}_{j \in [k]}$ and $\{\tf^k_j\}_{j \in [k+1:n]}$ are orthogonal to each other. Moreover,  $M\triangle M'\not= \emptyset$, for $ M\not=M' \in \P[d]$. 
Because of the product structure of the dictionary atoms spanning $\M$ this means that any $\M(M)$ and $\M(M')$ are mutually orthogonal, for $M\not=M'$.

The mutually orthogonal marginal linear subspaces $\{\M(M)\}_{M \in \P[d]}$ partition $\Rd$. The dimension of $\M(M)$ is given by $k^{d-\abs{M}} \prod_{i \in M} (n_i-k)$.
By the multi-binomial theorem it holds that
$$ \prod_{i=1}^d n_i = \sum_{M \in \P[d]} k^{d-\abs{M}} \prod_{i \in M} (n_i-k)$$
for $k \in [0: \min_{l \in [d]}n_l-1]$. This means that 
$$ \sum_{M \in \P[d]} {\rm dim}(\M(M))= n$$
and because $\{\M(M)\}_{M \in \P[d]}$ are mutually orhtogonal it follows that they also partition $\Rd$.

We can further partition any $\M(M)$ into $k^{d- \abs{M}}$ mutually orthogonal subspaces $\M(M,h), h \in [1:k]^{d-\abs{M}}$.

The partition results by defining the set of indices
$$ I^k_{M,h}:= \ttimes{i \in M} [k+1:n_i] \ttimes{i \not\in M} \{h_i\}$$
and the linear subspaces
$$ \M(M,h):= {\rm span} \{\tf^k_{k_1, \ldots, k_d}, (k_1, \ldots, k_d) \in I^k_{M,h}  \}.$$
Again, $\{\M(M,h)\}_{ h \in [k]^{d-\abs{M}}, M\in \P[d]} $ are mutually orthogonal and partition $\Rd$.

\begin{definition}[ANOVA decomposition]\label{vtv.d.5.1}
The decomposition of a tensor $f$ as

$$f= \sum_{M \in \P[d]} \sum_{h \in [1:k]^{d-\abs{M}}} f_{\M(M,h)}$$
is called ANOVA decomposition.
\end{definition}

By orthogonality we have that
$$ \norm{f}^2_2= \sum_{M \in \P[d]} \sum_{h \in [1:k]^{d-\abs{M}}} \norm{f_{\M(M,h)}}^2_2.$$

\subsection{Margins as lower dimensional objects}

Our aim is to apply a lower dimensional version of trend filtering to estimate $f^0_{\M(M,h)}$, for $M \not= \emptyset$. For $M = \emptyset$ it holds that $\abs{I^k_{M= \emptyset}}= k^d= \bigo(1)$. We will therefore estimate  $f^0_{\M(\emptyset,h)}$ by the least squares estimate $Y_{\M(\emptyset,h)}$  at the parametric rate $n^{-1}$.

To apply a lower dimensional version of trend filtering to estimate $f^0_{\M(M,h)}$ we first need to reinterpret $f_{\M(M,h)}$ as a $\abs{M}$-dimensional tensor. We then need to justify why we can apply Theorems \ref{vtv.t.2.3} and \ref{vtv.t.2.4} which require iid errors and are at the base of the adaptive rates by Theorem \ref{vtv.t.3.1} and the not-so-slow rates by Theorem \ref{vtv.t.2.2}.

By writing
$$ f_{\M(M,h)}= \barf_{\M(M,h)} \times  \ttimes{i \not\in M} \tf^k_{h_i},\ \barf_{\M(M,h)}\in  \R^{\ttimes{i \in M} n_i}$$
we can interpret $f_{\M(M,h)}$ as a $M$-dimensional object.

Similarly, we can write 
$$ Y_{\M(M,h)}= \bary_{\M(M,h)}\times  \ttimes{i \not\in M} \tf^k_{h_i},\  \bary_{\M(M,h)}\in  \R^{\ttimes{i \in M} n_i}.$$

Let $n_M:= \prod_{i \in M} n_i$. Because of the (partial) product structure of $f_{\M(M,h)}$ and since $\norm{\tf^k_{h_i}}^2_2= n_i, h_i \in [k]$ (cf. Definition \ref{vtv.d.2.5}), it holds that 

$$\norm{f_{\M(M,h)}}^2_2/n= \norm{\barf_{\M(M,h)}}^2_2/n_M.$$

Thanks to the above equation and to the ANOVA decomposition we can add up the rates of estimation of the margins to estimate the whole tensor.

\subsection{The estimator for the lower-dimensional margins}

For $M \in \P[d]\setminus \emptyset$ define
$$  D_M^k:= n_M^{k-1} \prod_{i \in M} D^k_i.$$

To estimate the whole tensor, we consider the estimator
$$ \hf= \sum_{M \in \P[d]} \sum_{h \in [k]^{d-\abs{M}}} \hf_{\M(M,h)},$$
where 
$$ \hf_{\M(M,h)}=\hat{ \barf}_{\M(M,h)} \ttimes{i \not\in M} \tf^k_{h_i}, \hat{ \barf}_{\M(M,h)}\in \R^{\ttimes{i \in M} n_i}.$$
We define
$$ \hf_{\M(\emptyset,h)}:=  \bary_{\emptyset,h} (\ttimes{i \in [d]}  \tf^k_{h_i}), \forall h \in [k]^d$$
and
$$ \hat{\barf}_{\M(M,h)}:=$$
$$:=\argmin_{\bar{\rf}_{\M(M,h)} \in  \R^{\ttimes{i \in M} n_i}} \left\{\norm{\bary_{\M(M,h)}-\bar{\rm f}_{\M(M,h)}}^2_2/n_M + 2 \lambda_{M,h} \1norm{D^k_M \bar{\rm f}_{\M(M,h)} } \right\},$$
where $\{ \lambda_{M,h}>0, h \in [k]^{d-\abs{M}}, M \in \P[d] \setminus \emptyset \}$ are positive tuning parameters.

We call $\1norm{D^k_M  \barf_{\M(M,h)}}$ the $k^{\rm th}$-order $\abs{M}$-dimensional Vitali total variation and $ \hat{\barf}_{\M(M,h)}$ the $\abs{M}$-dimensional trend filtering estimator.

\begin{remark}[We can apply Theorems \ref{vtv.t.2.3} and \ref{vtv.t.2.4}]
For $\barf \in \R^{\ttimes{i \in M} n_i} $ it holds that
\begin{eqnarray*}
 \bar{\epsilon}_{\M(M,h)} \odot \barf &=& \bare_{\M(M,h)} \odot \barf_{\M(M,h)}\\
 &=& \left(\sum_{{M' \subseteq M, h'_M=h}} (\bare_{\M(M',h')}\ttimes{i \in M\setminus M'} \tf^k_{h_i}) \right) \odot \barf_{\M(M,h)}.
\end{eqnarray*}
The $n_M$ entries of the tensor $\sum_{{M' \subseteq M, h'_M=h}} (\bare_{\M(M',h')}\ttimes{i \in M\setminus M'} \tf^k_{h_i})$ are the coefficients of the projection of $\epsilon$ onto the linear space $\ttimes{i \not\in M} \tf^k_{h_i} \ttimes{i \in M} \R^{n_i}$ and as such have i.i.d. $\N(0,\sigma^2n_M/n)$-distributed entries. We can therefore apply Theorems \ref{vtv.t.2.3} and \ref{vtv.t.2.4} with noise variance $\sigma^2 n_M/n$.
\end{remark}

\begin{remark}[Synthesis form for the estimator of lower dimensional margins]
The synthesis form of the estimator for the margins can be obtained in a similar way as for the $d$-dimensional margin (cf. Subsection \ref{vtv.ss.2.3}).
\end{remark}
\section{Denoising the whole tensor}\label{vtv.s.6}

We now put together the results from Sections \ref{vtv.s.3} and \ref{vtv.s.4} with the ANOVA decomposition given in Section \ref{vtv.s.5} to show adaptivity and not-so-slow rates for the estimation of the whole tensor.

\subsection{Adaptivity of trend filtering}

We fix $k \in \{1,2,3,4\}$.
By $S_{M,h}$ we denote a subset of $I^k_{M,h}$ satisfying the conditions for a hyperrectangular tessellation suitable for derivative matching. By $d_m^z(S_{M,h})$ we denote an analogon of the quantity $d^z_m$ appearing in Theorem \ref{vtv.t.3.1}, but defined on a hyperrectangular tessellation of $I^k_{M,h}$ generated by the enlarged version $\tS_{M,h}$ of the active set $S_{M,h}$.

The following theorem for denoising the whole tensor $Y$ by means of trend filtering holds.

\begin{theorem}[Adaptivity of tensor denoising with trend filtering]\label{vtv.t.6.1}
Choose $x,t>0$. Let $g \in \Rd$ be arbitrary.
For $M \not=\emptyset$ and a large enough constant $C>0$ only depending on $k$, choose
$$\lambda_{M,h} \ge C \ \abs{M} \ \sqrt{\sum_{i\in M} \left(\frac{d_{i,\max}(S_{M,h})}{n_i} \right)^{2k-1}} \lambda_0(t+d\log (k+1)) .$$
Then, with probability at least $1-e^{-x}-e^{-t}$, it holds that
\begin{eqnarray*}
&&\norm{\hf-f^0}^2_2/n \le \norm{g-f^0}^2_2/n + 4 \sum_{ \substack{M \in \P[d]\setminus \emptyset\\ h \in [k]^{d-\abs{M}}}} \lambda_{M,h} \norm{(D^k_M g)_{-S_{M,h}}}_1\\
&+& \frac{2\sigma^2}{n} \left(\sqrt{x+ d \log (k+1)}+ \sqrt{k^d} \right)^2\\
&+& \sum_{ \substack{M \in \P[d]\setminus \emptyset\\ h \in [k]^{d-\abs{M}}}} \frac{2\sigma^2}{n} \left(\sqrt{x+d\log (k+1)}+ \sqrt{ks_{M,h}} \right)^2\\
&+& \bigo\left( \sum_{ \substack{M \in \P[d]\setminus \emptyset\\ h \in [k]^{d-\abs{M}}}} \lambda_{M,h}^2 \ \left(\sum_{i\in M} \log(e d_{i,\max})\right) \sum_{m=1}^{s_{M,h}} \sum_{z \in \{-,+\}^{\abs{M}}} \left(\frac{n_M}{d_m^z} \right)^{2k-1} \right).
\end{eqnarray*}
In particular the constraint on $C$ is
$$ C \ge \frac{k^{\frac{2k-1}{2}}}{a_0} \text{ with } a_0= \begin{cases} 1, & k=1, \\ 8\sqrt{2}/7\approx 1.62, & k=2, \\144\sqrt{3}/76\approx 3.28 , & k=3, \\ 10.10, & k=4. \end{cases}$$
as
$$\min_{i \in M} \min_{m \in [s_{M,h}]}   \min_{h \in [k]^{d-\abs{M}}} \min \{d_{i,m}^-(S_{M,h}), d_{i,m}^+(S_{M,h}) \} \to \infty.$$
\end{theorem}

\begin{proof}
The result follows by the ANOVA decomposition. In total there are $(k+1)^d$ margins. As a consequence of the union bound, the result for the estimation of the whole tensor is attained with probability at least $1-e^{-t}-e^{-x}$ if in the application of Theorem \ref{vtv.t.3.1} one chooses $x+d\log(k+1)$ and $t+d\log(k+1)$ instead of $x$ and $t$ for some $x,t>0$.
\end{proof}

If $S_{M,h}$ are chosen to be regular grids and the tuning parameters are choosen as
$$ \lambda_{M,h} \asymp  \sqrt{\frac{\log n}{s^{\frac{2k-1}{\abs{M}}}n}}$$
 then the rate of Theorem \ref{vtv.t.6.1} is
 $$ \bigo\left( \frac{\log n}{n} \sum_{ \substack{M \in \P[d]\setminus \emptyset\\ h \in [k]^{d-\abs{M}}}} s_{M,h}^{\frac{2k(\abs{M}-1)+1}{\abs{M}}} \log(n_M/s_{M,h}) \right).$$
 
\subsection{Not-so-slow rates for trend filtering}
 
 We now present a not-so-slow rate of estimation for the whole tensor $f^0$ by trend filtering. We restrict again to tensors with $n_1 \asymp \ldots \asymp n_d \asymp n^{1/d}$.
 
 We let $c_{M,h}>0$ be constants of order $\bigo(1)$. The following theorem holds.
 
\begin{theorem}[Not-so-slow $\ell^1$-rate for trend filtering]\label{vtv.t.6.2}
Choose
$$\lambda \asymp n^{- \frac{H(d)+2k-1}{2H(d)+2k-1}} \log^{\frac{H(d)}{2H(d)+2k-1}}(n).$$
Then, with probability at least $1-\Theta(1/n)$, it holds that
\begin{eqnarray*}
\norm{\hf-f^0}^2_2/n &\le& \sum_{ \substack{M \in \P[d]\setminus \emptyset\\ h \in [k]^{d-\abs{M}}}} \min_{ \substack{\bar{\rm f}_{\M(M,h)}:\\ \1norm{D_M^k\bar{\rm f}_{\M(M,h)}}\le c_{M,h}}} \norm{\barf^0_{\M(M,h)}- \bar{\rm f}_{\M(M,h)}}^2_2/n_M\\
&& + \bigo\left( n^{- \frac{H(d)+2k-1}{2H(d)+2k-1}} \log^{\frac{H(d)}{2H(d)+2k-1}}n \right).
\end{eqnarray*}
\end{theorem}

\begin{proof}
We apply Theorem \ref{vtv.t.2.4} to $\hat{\barf}_{\M(M,h)}$ with $x+d\log(k+1)$ and $t+d\log(k+1)$ as  in the proof of Theorem \ref{vtv.t.6.1}. We choose $x\asymp t \asymp \log n$.

Let $\tS_{M,h}$ be an enlarged mesh grid.  We have to trade off with respect to $\tilde{s}_{M,h}\asymp s_{M,h}$ the terms 
$$ \frac{n_M \sigma^2}{n} \frac{s_{M,h}}{n_M} \asymp \underbrace{\frac{1}{s^{\frac{2k-1}{H(\abs{M})}}}}_{\asymp \tilde \gamma} \underbrace{\sigma \sqrt{\frac{n_M }{n} } \sqrt{\frac{\log n}{n_M}}}_{\asymp \lambda_0(\log n)}.$$
We therefore obtain the rate
$$ \bigo\left( n^{- \frac{H(\abs{M})+2k-1}{2H({\abs{M}})+2k-1}} \log^{\frac{H(\abs{M})}{2H(\abs{M})+2k-1}}n \right).$$
Since $\frac{H(\abs{M})}{2H(\abs{M})+2k-1}$ is decreasing in $\abs{M}$, the rate of estimation of the $d$-dimensional margin is limiting and we obain the claim.
\end{proof}
\section{Conclusion}\label{TV1S8}

We have shown that imposing structure to denoise $d$-dimensional tensors leads to an adaptive reconstruction. The structure is imposed via penalties on the $l$-dimensional $k^{\rm th}$-order Vitali TV of the $l$-dimensional margins of the tensor, for $l \in [d]$. If the tensor is a product of polynomials on a constant number of  hyperrectangles of any dimension $l \le d$, then the MSE is bounded as $ \norm{\hf - f^0}^2_2/n= \bigo(\log^2 n/n)$, with high probability. The true tensor $f^0$ can therefore be reconstructed at an almost parametric rate. The key aspects of our results are: the reformulation of the analysis estimator in synthesis form, the interpolating tensor to bound the effective sparsity and the ANOVA decomposition of a $d$-dimensional tensor. In the background of all our results there are the projection argmuents by \citet{dala17} to bound the random part of the problem, which are fundamental to prove the adaptativity of $\hf$ to the underlying unobserved $f^0$.

Note that we prove reconstruction for trend filtering of order $k= \{1,2,3,4\}$.  We are not able to prove that the approach we use to find an interpolating tensor for $k \in \{1,2,3,4\}$ gives a suitable interpolating tensor for general $k$.
Thus, although for each given finite $k$ we can check by computer whether our construction gives an interpolating vector, the problem remains open for general $k$.


\section*{Acknowledgements}
We would like to acknowledge support for this project
from the the Swiss National Science Foundation (SNF grant 200020\_169011).


%
%




\vskip 0.2in

\bibliography{library,/Users/fortelli/PhD/newlibrary/library}

\newpage
\appendix
\section{Oracle inequalities with fast and slow rates}
In this section we report an oracle inequality with fast rates and one with slow rates. These oracle inequalities correspond to the adaptive and to the non-adaptive bound of Theorem 2.2 in \cite{vand19}, see also  Theorems 2.1 and 2.2 in \citet{orte19-2} and Theorems 16 and 17 in \citet{orte19-4} adapted to have an enlarged active set.

\begin{theorem}[Oracle inequality with fast rates]\label{vtv.t.2.3}
Let $g \in \R^{n_1 \times \ldots \times n_d}$ and $S \subseteq \ttimes{i \in [d]} [k+2:n_i-k]$ be arbitrary.
For $x,t>0$, choose  $\lambda \ge {\tilde \gamma} \lambda_0(t)$.
Then, with probability at least $1-e^{-x}-e^{-t}$, it holds that 
\begin{eqnarray*}
\norm{(\hat{ f}-{ f^0})_{\N^{\perp}_k}}^2_2 /n &\le&  \norm{g-{ f^0}_{\N^{\perp}_k}}^2_2 /n + 4 \lambda \norm{(D^k g)_{-S}}_1\\
&&+ \left( \sigma\sqrt{\frac{2x}{n}} +\sigma \sqrt{\frac{k s}{n}} + \lambda \Gamma_{D^k}(S, v_{-S},q_S) \right)^2,
\end{eqnarray*}
where $q_S= \text{sign}((D^k g)_{S})$.
\end{theorem}

\begin{theorem}[Oracle inequality with slow rates]\label{vtv.t.2.4}
Let $g \in \Rd$ and $S \subseteq \ttimes{i \in [d]} [k+2:n_i-k]$ be arbitrary.
For $x,t>0$, choose $\lambda \ge {\tilde \gamma} \lambda_0(t)$.
Then, with probability at least $1-e^{-x}-e^{-t}$, it holds that
$$\norm{(\hf-f^0)_{\N^{\perp}_k}}^2_2 /n \le \norm{g- f^0_{\N^{\perp}_k}}^2_2 /n + 4 \lambda \norm{D^k g}_1 + \left( \sigma \sqrt{\frac{2x}{n}} + \sigma \sqrt{\frac{\tilde s}{n}}  \right)^2.$$
\end{theorem}

\section{Proofs of Section \ref{vtv.ss.2.3}}

\subsection{Proof of Lemma \ref{vtv.l.2.1}}\label{vtv.pl.2.1}
We prove Lemma  \ref{vtv.l.2.1} by induction.\\
{\bf Anchor: $k=1$}\\
Note that $\phi^1_1= \tf^1_1$ and $\phi^1_j- \tf^1_j= \alpha \phi^1_1$ for some $\alpha \in \R$. Therefore $D^1 \phi^1_1= D^1 \tf^1_1=0$ and $D^1(\phi^1_j- \tf^1_j)=0$. It follows that
$$ D^1 \tf^1_j= D^1 \phi^1_j= 1_{\{j'\ge j\}}- 1_{\{j'\ge j-1\}}= 1_{\{j\}}, j \in [2:n].$$\\
{\bf Step: $k-1$ implies $k$}\\
For $j \in [k-1]$ it holds that $D^k \phi^k_j=D^k \tf^k_j= D^k \phi^{k-1}_j= D^k \tf^{k-1}_j= 0$, since by assumption $D^{k-1}\phi^{k-1}_j= D^{k-1}\tf^{k-1}_j =0$ for $j \in [k-1]$.
Moreover
\begin{eqnarray*}
D^k \phi^k_j&=& D^k (\sum_{l \ge j} \phi^{k-1}_l)/n= D^1 (\sum_{l \ge j}D^{k-1} \phi^{k-1}_l)= D^1 \{1_{\{j' \ge j\}}\}_{j' \in [k:n]}\\
&=& \begin{cases} 0, & j =k,\\ 1_{\{j\}}, & j \in [k+1:n].\end{cases}
\end{eqnarray*}
It also holds that $ \phi^k_j- \tf^k_j= \sum_{l \in [k]} \alpha_l \phi_l^l, j \in [k:n]$ for some $\{\alpha_l \in \R\}_{l \in [k]}$ and therefore $D^k \phi^k_j= D^k \tf^k_j, j \in [k:n]$.
\hfill $\square$

\section{Proofs of Section \ref{vtv.s.3}}

\subsection{Proof of Lemma \ref{vtv.l.3.3}}\label{vtv.pl.3.3}

To bound the antiprojections we can use the dictionary $ \Phi^k$ instead of  $\tF^k$. Indeed, by Lemma 28 in \citet{orte19-4}, it holds that
$$\norm{{\rm A}_{\{\tf^k_{t}, t \in \tS\}} \tf^k_j}^2_2 \le \norm{{\rm A}_{\{\phi^k_{t}, t \in \tS\}} \phi^k_j}^2_2, \ j \in [k+1:n].$$

\subsubsection*{Bound on the antiprojections for $d=1$}

We first prove that, for $m=1, \ldots, s$,
$$ \norm{{\rm A}_{\tS} \tf^k_j}^2_2/n\le  \begin{dcases} \left(\frac{t_{m}-j}{n}  \right)^{2k-1}, & j \in R_{m}^-=[t_m^-:t_{m}],\\
 0, & j \in R_{m}^0=[t_{m}:t_{m}+k-1],\\
\left(\frac{j-t_{m}-k+1}{n}  \right)^{2k-1}, & j \in R_{m}^+=[t_{m}+k-1:t_{m}^+].\\
\end{dcases}$$
We then extend the reasoning to general dimension $d$.

For any $m \in [s]$, we  fix $j \in R^-_m$ and approximate $\phi^k_j$ by $\phi^k_{t_m}, \ldots, \phi^k_{t_m+k-1}$. By the definition of $\Phi^k$ we have that
$$ \phi^k_j(j')= n^{-k+1}(j'-j+1)^{k-1} 1_{\{j' \ge j\}}, \ j' \in [n].$$
Moreover note that for $k' \in \{0,1,\ldots, k-1\}$
\begin{equation}\label{vtv.e.3.5}
 \sum_{l=0}^{k'} (-1)^l \binom{k'}{l}  \phi^{k}_{t_m+l}= n^{-k'} \phi^{k-k'}_{t_m}= n^{-k+1} \{(j'-t_m+1)^{k-k'-1} 1_{\{j' \ge t_m\}}\}_{j'\in [n]}.\end{equation}
We now express $\phi^k_j$ as the sum of a linear combination of $\phi^k_{t_m}, \ldots,  \phi^k_{t_m+k-1}$ and a remainder. The linear combination will approximate the projection of $\phi^k_j$ onto $\{\phi^k_j, j \in \tS\}$, while the remainder will be an upper bound for the antiprojections.

For all $j'\in [n]$ it holds that
$$ \phi^k_j(j')= n^{-k+1}(j'-j+1)^{k-1}(1_{\{j\le j' \le t_m-1\}} +1_{\{j' \ge t_m\}}).$$
By the binomial theorem
\begin{eqnarray*}
&&(j'-t_m+1+t_m-j)^{k-1}1_{\{j' \ge t_m\}}\\
&&= \sum_{l=0}^{k-1} \binom{k-1}{l} (t_m-j)^{k-l-1} (j'-t_m+1)^l1_{\{j'\ge t_m\}}\\
&&= \sum_{l=0}^{k-1} \binom{k-1}{l} (t_m-j)^{k-l-1} n^l \phi^{l+1}_{t_m}.
\end{eqnarray*}
By Equation \eqref{vtv.e.3.5} we know that $\{\phi^{l+1}_{t_m}\}_{l\in [0:k-1]} \in {\rm span}(\{\phi^k_{t_m+l}\}_{l \in [0:k-1]})$.

Therefore, for $j \in R_m^-$,
\begin{eqnarray*}
 \norm{{\rm A}_{\tS} \tf^k_j}^2_2 &\le& n^{-2k+2} \sum_{j'=j}^{t_m-1} (j'-j+1)^{2k-2}\le n^{-2k+2} \int_{0}^{t_m-j} (j')^{2k-2} \ dj'\\
 &\le&  \frac{(t_m-j)^{2k-1}}{(2k-1)n^{2k-2}} \le n \left(\frac{t_m-j}{n}\right)^{2k-1}.
\end{eqnarray*}

Note that the construction of the partially orthonormalized dictionary $\tF^k$ can of course also be made starting from the collection of functions $\{1_{\{j \le j'\}} \}_{j \in [n]}, j' \in [n]$ instead of  $\{1_{\{j \ge j'\}} \}_{j \in [n]}, j' \in [n]$, cf. Definition  \ref{vtv.d.2.5}. The resulting dictionaries $\tF^k$ coincide, up to permutation of the column indices. As a consequence, the calculation we  showed to approximate $\norm{{\rm A}_{\tS} \tf^k_j}^2_2$ for $j \in R^-_m$ can be carried out with the dictionary $\tF^k$ based on  $\{1_{\{j \le j'\}} \}_{j \in [n]}, j' \in [n]$ to obtain the approximation
$$\norm{{\rm A}_{\tS} \tf^k_j}^2_2\le n \left(\frac{j-t_m-k+1}{n}\right)^{2k-1} ,  \ j \in R^+_m.$$ This consideration also applies in higher-dimensional situations. 

\subsubsection*{Bound on the antiprojections for general dimension $d$}

By the same reasons as above, we consider without loss of generality $(k_1, \ldots, k_d) \in R^{-, \dots, -}_m$.
We decompose $\phi^k_{k_1, \ldots, k_d}$ as follows
$$ \phi^k_{k_1, \ldots, k_d}(j_1, \ldots, j_d)= n^{-k+1} \prod_{i=1}^d (a_i(j_i)+ b_i(j_i)),$$
for $j_i \in [n_i], i \in [d]$, where
\begin{align*}
a_i=a_i(j_i)&= (j_i-k_i+1)^{k-1} 1_{\{k_i \le j_i \le t_{i,m}-1\}},\\
b_i=b_i(j_i)&= (j_i-k_i+1)^{k-1} 1_{\{ j_i \ge t_{i,m}\}},\\
c_i=c_i(j_i)&= (j_i-k+1)^{k-1} 1_{\{ j_i \ge k\}} \ge a_i+ b_i.
\end{align*}
Note that $a_i,b_i$ depend on $t_{i,m}$, while $c_i$ does not. Moreover, for all  $(l_1, \ldots, l_d) \in [0,k-1]^d$ it holds that $ \ttimes{i \in [d]} \{t_{i,m+l_i}\} \in \tS$.
Thus, we approximate
$$ \norm{{\rm A}_{\tS} \tf^k_{k_1, \ldots, k_d}}^2_2 \le n^{-2k+2} \sum_{1, \ldots, 1}^{n_1, \ldots, n_d} \left( \prod_{i=1}^d (a_i+b_i)- \prod_{i =1}^d b_i \right)^2,$$
since by Equation \eqref{vtv.e.3.5} the contributions of $\prod_{i =1}^d b_i$ are spanned by $\phi^k_S$.
Note that $\prod_{i=1}^d (a_i+b_i)- \prod_{i =1}^d b_i$ is nonzero on
$$ \ttimes{i \in [d]} [k_i:n_i] \setminus \ttimes{i \in [d]} [t_{i,m}: n_i] \subseteq \cup_{i \in [d]} \left( [k_i:t_{i,m}-1] \times \ttimes{l \not=i} [1:n_l]   \right).$$
Moreover, on $[k_i:t_{i,m}-1] \times \ttimes{l \not=i} [1:n_l] $, it holds that $\prod_{i=1}^d (a_i+b_i)- \prod_{i =1}^d b_i \le a_i \prod_{l \not=i} c_l$.
Therefore
$$ \norm{{\rm A}_{\tS} \tf^k_{k_1, \ldots, k_d}}^2_2 \le n^{-2k+2} \sum_{i=1}^d \sum_{1, \ldots, 1}^{n_1, \ldots, n_d} \left( a_i^2(j_1) \prod_{l \not=i} c_l^2(j_l) \right).$$
As in the one-dimensional case, $ n_i^{-2k+2}\sum_{j_i=1}^{n_i} a_i^2(j_i)\le n_i \left(\frac{t_i-k_i}{n_i} \right)^{2k-1}$ and $ n_i^{-2k+2}\sum_{j_i=1}^{n_i} c_i^2(j_i)\le n_i $. It follows that
$$ \norm{{\rm A}_{\tS} \tf^k_{k_1, \ldots, k_d}}^2_2 \le n \sum_{i=1}^d \left(\frac{t_i-k_i}{n_i} \right)^{2k-1}.$$

Note that as soon as $j_i \in R_{i,m}^0$ for some coordinate $i \in [d]$, then $a_i(j_i)=0$ and the $i^{\rm th}$ coordinate does not contribute to the antiprojections. The bounds for all other hyperrectangles $R^z_m, z \in \{-,0,+\}^d$ follow by analogous calculations. \hfill $\square$

\subsection{Proof of Lemma \ref{vtv.l.3.4}}\label{vtv.pl.3.4}
For any $m \in [s]$ and for any $(j_1, \ldots, j_d) \in R_m$ it holds that
\begin{eqnarray*}
\sqrt { \sum_{l=1}^d {\tilde v}^2_{i,m}(j_i)}&\le& \sum_{l=1}^d  {\tilde v}_{i,m}(j_i) \le 
\sum_{l=1}^d  { v}_{i,m}(j_i)  \left( { \max\{d_{i,m}^-, d_{i,m}^+ \} \over n_i}\right)^{\frac{2k-1}{2}}\\
&\le& \sum_{l=1}^d  { v}_{i,m}(j_i)  \sqrt{ \sum_{i=1}^{d} \left( {\max\{d_{i,m}^-, d_{i,m}^+ \} \over n_i}\right)^{{2k-1}}}\le v_{j_1, \ldots, j_d} {\tilde \gamma} .
\end{eqnarray*} \hfill $\square$

\subsection{Proof of Lemma \ref{vtv.l.3.5}}\label{vtv.pl.3.5}
Fix $i \in [d]$ and $m \in [s]$. Say $q_{t_m}=1$.  Since $w_{i,l,m} \in [0,1], l\not=i$, for any $j_i \in R_{i,m}^- \cup R_{i,m}^0 \cup R_{i,m}^+  $  it holds that
$$  \prod_{l =1}^d w_{i,l,m}(j_l) \le   \left( 1- \frac{\sqrt{v_{i,m}(j_i)}}{C} \right) \prod_{l \not= i} w_{i,l,m}(j_l) \le \left( 1- \frac{\sqrt{v_{i,m}(j_i)}}{C} \right).$$
Moreover, for any $(j_1, \ldots, j_d) \in R_{m}$  it holds that
\begin{eqnarray*}
w_{j_1, \ldots, j_d}&=&\frac{1}{d}\sum_{i=1}^d \prod_{l =1}^d w_{i,l,m}(j_l)\\
&\le&   \frac{1}{d} \sum_{i=1}^d \left( 1- \frac{\sqrt{v_{i,m}(j_i)}}{C} \right)= 1- \sum_{i=1}^d \frac{\sqrt{v_{i,m}(j_i)}}{d C}= 1-v_{j_1, \ldots, j_d}.
\end{eqnarray*}
Analogous expressions hold if $q_{t_m}=-1$.
The claim follows by  noting that the conditions of the definition of interpolating tensor (Definition \ref{vtv.d.2.12}) are satisfied for $w$. \hfill $\square$

\subsection{Matching derivatives}\label{vtv.match.derivatives}

To obtain continuous vectors with $k-1$ continuous derivatives and piecewise constant $k^{\rm th}$ derivative, we split $[0,1]$ into $N_{\omega}$, resp. $N_{\rw}$, intervals of equal length, where $N_{\omega}=k$, $N_{\rw}=k+1$ if $k$ is odd and $N_{\rw}=k+2$ if $k$ is even. We denote these intervals by $\{[x_{l-1}, x_l]\}_{l=1}^{N_{\{\omega, \rw\}}}$ with $x_0=0$ and $x_{N_{\{\omega, \rw\}}}=1$. We choose
$$
\omega(x)=\begin{dcases}
1-a_0x^{\frac{2k-1}{2}}, & x \in [x_0,x_1],\\
b_{l,k}x^k+ b_{l,k-1}x^{k-1}+ \ldots + b_{l,1}x+b_{l,0}, & x \in [x_{l-1},x_l],\\
& l \in [2:k-1],\\
c_0(1-x)^k, & x \in [x_{k-1}, x_k].
\end{dcases}$$
We moreover choose 
$$
\rw(x)=\begin{dcases}
1-{\rm a}_0x^{k}, & x \in [x_0,x_1],\\
{\rm b}_{l,k}x^k+ {\rm b}_{l,k-1}x^{k-1}+ \ldots + {\rm b}_{l,1}x+{\rm b}_{l,0}, & x \in [x_{l-1},x_l],\\
&l \in [2:N_{\rw}/2-1],\\
{\rm a}_L (1/2-x)^L+ \ldots + {\rm a}_1(1/2-x)+1/2, & x \in [x_{N_{\rw}/2-1}, x_{N_{\rw}/2+1}],
\end{dcases}$$
where $L=k-1$ if $k$ is even and $L=k$ if $k$ is odd.

We choose both the coefficients ($a_0, a_L, \ldots, a_1, \{b_{l,k}, \ldots, b_{l,0}\}_{l}, c_0$) and (${\rm a}_0, {\rm a}_L, \ldots, {\rm a}_1, \{{\rm b}_{l,k}, \ldots, {\rm b}_{l,0}\}_{l}$) by derivative matching. We require the $k-1$ derivatives of the different pieces of the interpolating polynomials to match at the junctions between the intervals. This gives place to piecewise constant $k^{\rm th}$ derivatives with the exception of the interval $[x_0,x_1]$ where $ \omega^{(k)}(x)\asymp - 1/\sqrt{x}$.

 Matching derivatives for $\omega$ means solving a system of $k(k-1)$ equations and $k(k-1)$ unknowns. Matching derivatives for $\rw$ means solving a system of $k(k/2)$ equations and $k(k/2)$ unknowns when $k$ is even and $k(k-1)/2$ equations and $k(k-1)/2$ unknowns when $k$ is odd.
 We therefore do not need to do any derivative matching for $k=1$, where we just take $\omega(x)=1-\sqrt{x}$ and $\rw(x)=1-x$. 
  
As an alternative to discretizing a continuous version of the interpolating polynomials, one can also proceed by matching discrete differences. The two approaches are equivalent when $\min_{i \in [d]} \min_{m \in [s]}  \min \{d_{i,m}^-, d_{i,m}^+ \} \to \infty$ as $n \to \infty$. Discrete derivative matching requires that the counterpart of each interval $[x_{l-1}:x_l]$ contains at least $k$ points. We therefore require that
 $$ \min\{d_{i,m}^-, d_{i,m}^+ \} \ge (k+2)k, \ \forall i \in [d], \ \forall m \in [s].$$
 We refer to \citet{vand19} for details on discrete derivative matching.

\subsection{Partial integration}\label{vtv.pl.3.6}

Some consequences of the fact that both the resulting $\omega$ and $\rw$ have piecewise constant $k^{\rm th}$ derivatives with the exception of the interval $[0,x_1]$ where $ \omega^{(k)}(x)\asymp - 1/\sqrt{x}$ are shown in the next lemma, which is be useful to compute the bound on the effective sparsity in Lemma \ref{vtv.l.3.7}.

\begin{lemma}[Discrete differences of some polynomials]\label{vtv.l.3.6}
Let for some $d \in {\mathbb N}$, $d \ge 2k$, 
$${\rm q}_j:= (j/d)^{2k-1 \over 2} , \ j= 0 , \ldots , d . $$
Then
$$ n^{-2k+2} \| D^k {\rm q} \|_2^2 = \bigo(  \log (ed)/d^{2k-1}) . $$
Let for some $d \in {\mathbb N}$, $d \ge 2k$, 
$${\rm p}_j:= (j/d)^{k} , \ j= 0 , \ldots , d . $$
Then 
$$ n^{-2k+2}\| D^k {\rm p} \|_2^2 =\bigo(1/d^{2k-1}) . $$
\end{lemma}

\begin{proof}
We have for $j \ge k$
\begin{eqnarray*}
 n^{-2k+2}(D^k \rmq)_j& =& \sum_{l=0}^k { k \choose l} (-1)^l \left(\frac{j-l}{d}\right)^{2k-1 \over 2}\\
 &=& \left(\frac{j}{d}\right)^{2k-1 \over 2} \biggl [ \sum_{l=0}^k {k \choose l} (-1)^l \biggl (1- { l \over j} \biggr )^{2k-1\over 2} \biggr ] . 
\end{eqnarray*} 
We do a $(k-1)$-term Taylor expansion of $x \mapsto (1- x)^{2k-1 \over 2} $ around $x=0$: 
$$ (1- x)^{2k-1 \over 2} =\sum_{i=0}^{k-1} a_i x^i + {\rm rem} (x) ,$$
where $a_0=1$, $a_1= -{2k-1 \over 2} , \ldots, a_{k-1}$ are the coefficients of the
Taylor expansion and where the remainder ${\rm rem} (x)$ satisfies 
$$ \sup_{0 \le x \le 1/2 } | {\rm rem} (x) | = \bigo\left( | x |^k\right) . $$
Thus
 \begin{equation*}
 \sum_{l=0}^k { k \choose l} (-1)^l\biggl (1- { l \over j} \biggr )^{2k-1\over 2}  = \sum_{l=0}^k { k \choose l} (-1)^l \left( \sum_{i=0}^{k-1} a_i \biggl( { l\over j } \biggr)^i +  {\rm rem} \biggl( {l \over j} \biggr) \right),
\end{equation*} 
where
$$ \sum_{l=0}^k { k \choose l} (-1)^l \sum_{i=0}^{k-1} a_i \biggl( { l\over j } \biggr)^i=0$$
since
$$  \left\{ \sum_{i=0}^{k-1} { a_i  } \biggl ( { l\over j } \biggr )^i  \right\}_{l=0}^{k}$$
is a polynomial of degree $k-1$ and hence its $k^{\rm th}$-order differences are zero.
 It follows that for $j \ge k$, 
$$ \biggl | \sum_{l=0}^k { k\choose l} (-1)^l\biggl (1- { l \over j} \biggr )^{2k-1\over 2}\biggr | \le 
 \sum_{l=0}^k  { k \choose l} \biggl | {\rm rem} \biggl ( {l \over j} \biggr ) \biggr |    = \bigo \left( {1 \over j^k }\right) .$$
Then for $j \ge k$, $ n^{-2k+2}(D^k \rmq)_j  = \bigo \left( 1/({j^{\frac{1}{2}} d^{\frac{2k-1}{2}}})\right) $.
 So $n^{-2k+2}\norm{D^k\rmq}^2_2 = \bigo \left( {\log(ed)}/{ d^{{2k-1}}}\right) $.
 
For $\rp$ the same arguments go through. We obtain that $(D^k\rp)_j=\bigo(1/d^{k})$ and so $n^{-2k+2}\norm{D^k\rp}^2_2 = \bigo \left( 1/{ d^{{2k-1}}}\right)$.

 \end{proof}

 \subsection{Proof of Lemma \ref{vtv.l.3.7}}\label{vtv.pl.3.7}

We prove a bound on the effective sparsity holding for every sign configuration. We eliminate the dependence on the sign configuration by decoupling partial integration on the whole interpolating tensor ($\norm{(D^k)'w}^2_2$) into taking $k^{\rm th}$-order differences on the hyperrectangles  $\{R_m\}_{m =1}^s$ ($\norm{D^k w(R_m)}^2_2$, where $w(R_m)=\{w_{j_1, \ldots, j_d}\}_{(j_1, \ldots, j_d)\in R_m}$ denotes the restriction of the interpolating tensor $w$ to the set of indices $R_m$). 

To do this, we define the boundaries $B(R_m)$ of a rectangle $R_m$ as
$$ B(R_m):= R_m \setminus \ttimes{i \in [d]} [t_{i,m}^-+k: t_{i,m}^+-k] .$$

It holds that
$$ n^{-2k+1} \norm{(D^k)'w}^2_2= \bigo \left( \sum_{m=1}^s \left( n^{-2k+1}\norm{D^k w(R_m)}^2_2 + \norm{w(B(R_m))}^2_2 \right) \right).$$
By the definition of the interpolating tensor $w$ it holds that
\begin{eqnarray*}
&&n^{-2k+1}\norm{D^k w(R_m)}^2_2\\
&&= \bigo\left(n^{-2k+1} \sum_{i=1}^d \norm{D^k \ttimes{l \in [d]} w_{l,i,m} }^2_2 \right)\\
&&= \bigo\left(n^{-2k+1} \sum_{i=1}^d \prod_{l=1}^d \norm{D^k  w_{l,i,m} }^2_2 \right)\\
&&= \bigo\left( \sum_{i=1}^d \prod_{l=1}^d \left( n_l^{-2k+1}\norm{D^k  w_{l,i,m}^- }^2_2 + \sum_{j_l=t_{l,m}-k}^{t_{l,m}-1} (1-w_{l,i,m}^-(j_l))^2\right. \right. \\
&& \left. \left. + n_l^{-2k+1}\norm{D^k  w_{l,i,m}^+ }^2_2 + \sum_{j_l=t_{l,m}+k}^{t_{l,m}+2k-1} (1-w_{l,i,m}^+(j_l))^2 \right)\right),\\
\end{eqnarray*}
where the sums stem from the differences involving the constant part of $w$ on $R^0_{l,m}$.
Because of the form chosen for $\omega$ and $\rw$, it holds that
$$ \sum_{j_l=t_{l,m}-k}^{t_{l,m}-1} (1-w_{l,i,m}^-(j_l))^2 
= \begin{cases} \bigo(\omega^2(1/d_{i,m}^-)) \\ \bigo(\rw^2(1/d_{l,m}^-)) \end{cases} 
= \begin{cases} \bigo(1/(d_{i,m}^-)^{2k-1}), & l = i, \\ \bigo(1/(d_{l,m}^-)^{2k}), & l \not= i. \end{cases}$$
A similar bound holds for $\sum_{j_l=t_{l,m}+k}^{t_{l,m}+2k-1} (1-w_{l,i,m}^+(j_l))^2$.
By Lemma \ref{vtv.l.3.6} it holds that
$$ n_l^{-2k+1}\norm{D^k  w_{l,i,m}^- }^2_2 = \begin{cases} \bigo(\log (e d_{i,m}^-)/(d_{i,m}^-)^{2k-1}), & l = i, \\ \bigo(1/(d_{l,m}^-)^{2k-1}), & l \not= i. \end{cases}$$
A similar bound holds for $n_l^{-2k+1}\norm{D^k  w_{l,i,m}^+ }^2_2$.

We now just have to upper bound the contributions of the boundaries $B(R_m)$. For $k=1$, $w(B(R_m))=0$, for all $m \in [s]$ and the boundaries do not contribute to the effective sparsity. For $k\ge 2$ it holds that
$$ \sum_{B(R_m)} w^2_{j_1, \ldots, j_d}= \bigo \left( \sum_{i=1}^d \sum_{ B(R_m)} \prod_{l=1}^d w^2_{l,i,m}(j_l) \right)= \bigo \left( \sum_{i=1}^d \sum_{z \in \{-,+\}^d}  \frac{1}{(d_m^{z})^{2k-1}}\right)$$
since all the contributions on the boundaries have the same dependence on $k$ and we can approximate the volume of the boundaries by the sum of the volume of the 
$2^d$ fractions $\{R_m^{z}\}_{ z \in \{-,+\}^d}$ of the hyperrectangle.

It therefore holds that
\begin{equation*}
n^{-2k+1} \norm{(D^k)'w}^2_2=\bigo \left(  \left(\sum_{i=1}^d \log( e d_{i, \max}(S))\right) \sum_{m=1}^s  \sum_{z \in \{-,+\}^d} \frac{1}{(d_m^z)^{2k-1}} \right)
\end{equation*}
and the claim follows.

\hfill $ \square$

\section{Proofs of Section \ref{vtv.s.4}}

\subsection{Proof of Lemma \ref{vtv.l.4.1}}\label{vtv.pl.4.1}

\subsubsection*{Setting}
To calculate the inverse scaling factor when the active set is an enlarged mesh grid $\tS$, we decompose a dictionary atom -- which is a product of sums -- into a sum of products. Some of the components will be spanned by the dictionary atoms indexed by the mesh grid. The remaining components will contribute to the antiprojection. 

By Lemma 28 in \citet{orte19-4} we can look at the dictionary atoms ${\phi^k_{j_1, \ldots, j_d}}$ instead of ${\tf^k_{j_1, \ldots, j_d}}$, see also the proof of Lemma \ref{vtv.l.3.3} in Appendix \ref{vtv.pl.3.3}.

We therefore consider
$$ \phi^k_{j_1, \ldots, j_d}= \phi^k_{j_1} \times \ldots \times \phi^k_{j_d},$$
where, for $i \in [d]$
$$ \phi^k_{j_i}= n_i^{-k+1} (j-j_i+1)^{k+1} 1_{\{j \ge j_i\}}.$$

\subsubsection*{Projection of the mesh grid on single coordinates}
Now choose $z_{i,l} \in Z_i(l)$ such that $j_i \le z_{i,1} \le \ldots \le z_{i,d-1} \le z_{i,d}$.
By the definition of the mesh grid we can choose $z_{i,l} \in Z_i(l)$ such that
\begin{itemize}
\item $\abs{j_i-z_{i,1}} = \bigo( n_i/s^{\frac{1}{H(d)}})$;
\item $\abs{z_{i,l}-z_{i,l-1}} = \bigo( n_i/s^{\frac{1}{lH(d)}})$, $l \in [2:d] $;
\item $\abs{z_{i,d}}\le n_i $.
\end{itemize}

\subsubsection*{The decomposition}
We now decompose the factors into sums:
$$ \phi^k_{j_i}= \sum_{l=0}^{d} u_{i,l},$$
where, for $j \in [n_i]$,
\begin{align*}
u_{i,0}&:= 1_{\{  j \in[j_i: z_{i,1}-1]\}} n_i^{-k+1} (j-j_i+1)^{k-1},\\
u_{i,l}&:= 1_{\{ j \in [z_{i,l}:z_{i,l+1}-1] \}} n_i^{-k+1} (j-j_i+1)^{k-1}, & l \in [1:d-1],\\
u_{i,d}&:= 1_{\{ j \in [z_{i,d}:n_i] \}} n_i^{-k+1} (j-j_i+1)^{k-1},\\
\end{align*}
Note that $\{u_{i,l}\}_{l=0}^d$ are mutually orthogonal.

Thanks to the decomposition of the factors, the following decomposition of the dictionary atom $\phi^k_{j_1, \ldots, j_d}$ holds:
$$ \phi^k_{j_1, \ldots, j_d}= \sum_{(l_1, \ldots, l_d)\in [0:d]^d} \prod_{i=1}^d u_{i,l_i},$$
where $\{\prod_{i=1}^d u_{i,l_i}\}_{(l_1, \ldots, l_d)\in [0:d]^d}$ are mutually orthogonal. We therefore obtain a decomposition of a product of sums into a sum of products.

\subsubsection*{Partitioning the decomposition}
We now partition $\{(l_1, \ldots, l_d)\in [0:d]^d\}$ into two subsets: $\Sigma$ and $\Sigma^c$.
Define
$$ \Sigma:= \{(l_1, \ldots, l_d) \in [0:d]^d: \abs{\{i \in [d]: l_i\le z \}} \le z, \forall z \in [0:d]\}.$$
This means that $\Sigma$ contains tuples $(l_1, \ldots, l_d)$ having at most $d$ entries with value at most $d$ {\it and} at most $d-1$ entries with value at most $d-1$ {\it and} ... {\it and} at most $1$ entry with value at most $1$ {\it and} no entry with value $0$.

\subsubsection*{Connecting the decomposition with the enlarged mesh grid}
We now want to show that, for any $(l_1, \ldots, l_d) \in \Sigma$, $\prod_{i=1}^d u_{i,l_i}$ can be obtained as a linear combination of $\{\phi^k_{j_1, \ldots, j_d}\}_{ (j_1, \ldots, j_d)\in \tS }$. These components will approximate the projection of any $\phi^k_{j_1, \ldots, j_d}$ onto the linear span of $\{\phi^k_{j_1, \ldots, j_d}\}_{(j_1, \ldots, j_d)\in \tS }$.

For $l_i\in [1:d-1]$ it holds that
$$ u_{i,l_i}(j)= 1_{\{z_{i,l_i}\le j \}} n_i^{-k+1} (j-j_i+1)^{k-1} - 1_{\{z_{i,l_i+1}\le j \}} n_i^{-k+1} (j-j_i+1)^{k-1}.$$
In analogy to the proof of Lemma \ref{vtv.l.3.3} (use the binomial theorem and Equation \eqref{vtv.e.3.5}) it holds that $u_{i,l_i} \in {\rm span} (\{\phi^k_{z_{i,l_i}+h}\}_{h=0}^{k-1} \cup \{\phi^k_{z_{i,l_i+1}+h}\}_{h=0}^{k-1} )$.

For $l_i \in [d]$ it holds that $u_{i,d} \in {\rm span} (\{\phi^k_{z_{i,l_i}+h}\}_{h=0}^{k-1})$

\subsubsection*{We need a claim}
We now show that
$$ (l_1, \ldots, l_d) \in \Sigma \Longrightarrow (l_1',\ldots, l_d') \in \Sigma,$$
where $l_i' \ge l_i, \forall i \in [d]$ by proving that
$$ (l_1, \ldots, l_d) \in \Sigma \Longrightarrow (l_1,\ldots,l_{d-1}, l_d+1) \in \Sigma,$$
where without loss of generality we choose the index $l_d$ and assume that $l_d\le d-1$.

As a consequence it will follow that, for any $(l_1, \ldots, l_d) \in \Sigma$, $\prod_{i=1}^d u_{i,l_i}$ can be obtained as a linear combination of $\{\phi^k_{j_1, \ldots, j_d}\}_{ (j_1, \ldots, j_d)\in \tS }$.

We now prove the claim: assume that $(l_1, \ldots, l_d)\in \Sigma$, i.e., $\abs{\{i \in [d]: l_i\le z \}}\le z, \forall z \in [0:d]$. Take $(l_1', \ldots, l_d')$ as $l_i'=l_i, i \in [d-1]$ and $l_d'=l_d+1$. Then
\begin{eqnarray*}
\abs{\{i \in [d]: l'_i\le z \}}&=& \abs{\{i \in [d-1]: l_i\le z \}}+ 1_{\{z \ge l_d+1 \}}\\
&\le & z - 1_{\{z \ge l_d \}} + 1_{\{z \ge l_d+1 \}}\le z.
\end{eqnarray*}
Therefore $(l_1', \ldots, l_d')\in \Sigma$ and the claim is proved.

\subsubsection*{Approximating the antiprojections}
Thanks to the above claim and to the mutual orthogonality of the elements of $\{\prod_{i=1}^d u_{i,l_i}\}_{(l_1, \ldots, l_d)\in [0:d]^d}$, we can approximate as follows:
$$ \norm{{\rm A}_{\tS} \phi^k_{j_1, \ldots, j_d}}^2_2/n \le \sum_{(l_1, \ldots, l_d)\not\in \Sigma} \norm{\prod_{i=1}^d u_{i,l_i}}^2_2/n = \sum_{(l_1, \ldots, l_d)\not\in \Sigma} \prod_{i=1}^d \norm{ u_{i,l_i}}^2_2/n_i.$$
Now we use the following property:
$$ \norm{ u_{i,l_i}}^2_2/n_i= \bigo(s^{-\frac{2k-1}{(l_i+1)H(d)}} ).$$
The larger $l_i$, the larger the contribution of $ \norm{u_{i,l_i}}^2_n$.

It therefore only remains to find the order of the largest contribution(s) indexed by $\Sigma^c$.
A tuple of indices in $\Sigma^c$ giving the contribution highest in order is
$$ (d-1, \ldots, d-1).$$
It holds that
$$ \norm{{\rm A}_S\psi_{j_1, \ldots, j_d}}^2_2/n = \bigo\left(\prod_{i=1}^d \norm{u_{i,l_i}}^2_n \right)= \bigo( s^{-\frac{2k-1}{H(d)}}).$$
Since the upper bound does not depend on $(j_1, \ldots, j_d)$ we read directly that
$$ {\tilde \gamma} = \bigo \left(s^{- \frac{2k-1}{2H(d)}} \right).$$
 \hfill $\square$

\end{document}